\RequirePackage[final]{graphicx}
\documentclass[oneside]{amsart}

\usepackage[english]{babel}
\usepackage[T1]{fontenc}
\usepackage[utf8]{inputenc}

\usepackage{amsmath,
		     amssymb,
		     amsfonts,
		     amsthm}
\usepackage{algorithmic}
\usepackage{stmaryrd} %
\usepackage{mathrsfs} %
\usepackage{dsfont} %

\usepackage{mathtools}
\usepackage{enumerate}
\usepackage{todonotes}
\usepackage{color}
\usepackage{ifdraft}
\usepackage{xcolor}

\numberwithin{equation}{section}

\usepackage{etoolbox}

\usepackage{marginnote}
\makeatletter
\renewcommand{\@todonotes@drawMarginNoteWithLine}{%
\begin{tikzpicture}[remember picture, overlay, baseline=-0.75ex]%
    \node [coordinate] (inText) {};%
\end{tikzpicture}%
\marginnote[{%
    \@todonotes@drawMarginNote%
    \@todonotes@drawLineToLeftMargin%
}]{%
    \@todonotes@drawMarginNote%
    \@todonotes@drawLineToRightMargin%
}%
}
\makeatother
\usepackage{soul}
\makeatletter
\if@todonotes@disabled

\else

\fi
\makeatother

\makeatletter
\patchcmd{\@settitle}{\uppercasenonmath\@title}{\LARGE\scshape}{}{}
\patchcmd{\@setauthors}{\MakeUppercase}{\scshape}{}{}
\patchcmd{\@setauthors}{\footnotesize}{\normalsize}{}{}
\makeatother

\usepackage
	[a4paper,
	 text={145mm,240mm},
	 head=6mm,
     \ifdraft{marginparsep=1pt,
              marginparwidth=27mm,
              showframe} %
             {marginparwidth=22mm}
	]{geometry}

\usepackage[obeyspaces,hyphens,spaces]{url}
\ifdraft{
	\usepackage[notcite,
    			notref]{showkeys}
	\usepackage{seqsplit}
	
	}{}

\definecolor{darkblue}{rgb}{0.0, 0.2, 0.6}
\definecolor{brickred}{rgb}{0.8, 0.25, 0.33}
\definecolor{forestgreen}{rgb}{0.13, 0.55, 0.13}
\usepackage[colorlinks,linkcolor=brickred,citecolor=forestgreen,urlcolor=darkblue,hypertexnames=true]{hyperref}

\usepackage{thmtools}
\declaretheoremstyle[spaceabove=8pt,
					 spacebelow=8pt]{mythmstyle}

\declaretheorem[name=Theorem,
				numberwithin=section,
               ]{theorem}
\declaretheorem[name=Proposition,
				sibling=theorem,
               ]{proposition}

\declaretheorem[name=Assumption,
				sibling=theorem
               ]{assumption}

\declaretheorem[name=Remark,
				style=remark,
                qed=$\triangle$,
                sibling=theorem]{remark}
\declaretheorem[name=Example,
				style=remark,
                qed=$\triangle$,
                sibling=theorem]{example}

\usepackage{csquotes}
\usepackage[backend=biber,
		    style=numeric-comp,
		    url = true,
		    isbn = false,
		    doi = true,
		    eprint = true,
            	   maxnames=10]{biblatex} 
\usepackage{mathmacros}

\title[Micro-macro acceleration for linear slow-fast SDEs]{Study of micro-macro acceleration schemes for linear slow-fast stochastic differential equations with additive noise$^*$\footnote{$^*$This is a pre-print of an article published in BIT Numerical Mathematics. The final authenticated version is available online at: https://doi.org/10.1007/s10543-020-00804-5}}

\subjclass[2010]{Primary, 65C30, 60H35, 65L20; Secondary, 94A17, 62E17}

\keywords{micro-macro simulations, entropy optimisation, stiff stochastic differential equations, Kullback-Leibler divergence, stability}

\author[K.~Debrabant]{Kristian Debrabant}
\address[K.~Debrabant]{Department of Mathematics and Computer Science, University of Southern Denmark, Campusvej 55, 5230 Odense M, Denmark}
\email{debrabant@imada.sdu.dk}

\author[G.~Samaey]{Giovanni Samaey}
\address[G.~Samaey]{KU Leuven, Department of Computer Science, NUMA Section, 3001 Heverlee, Belgium}
\email{giovanni.samaey@kuleuven.be}

\author[P.~Zieli\'{n}ski]{Przemys{\l}aw Zieli\'{n}ski}
\address[P.~Zieli\'{n}ski]{KU Leuven, Department of Computer Science, NUMA Section, 3001 Heverlee, Belgium}
\curraddr{Ecole F\'{e}d\'{e}rale Polytechnique de Lausanne, CH-1015 Lausanne, Switzerland}
\email{przemyslaw.zielinski@epfl.ch}

\date{\today}

\newcommand{\prob}{\mathscr{P}} %
\newcommand{\contb}{\mathscr{C}_{\mathrm{b}}}
\newcommand{\leb}[1]{\mathscr{L}^#1}
\newcommand{\bor}{\mathfrak{B}} %
\newcommand{\Law}[1]{\mathrm{Law}(#1)} %
\newcommand{\nd}{\mathcal{N}} %
\newcommand{\diag}{\mathrm{diag}} %
\newcommand{\kld}{\mathcal{D}} %
\newcommand{\gen}{\mathcal{L}} %
\newcommand{\match}{\mathcal{M}} %
\newcommand{\res}{\mathcal{R}} %
\newcommand{\prop}{\mathcal{S}} %
\newcommand{\ext}{\mathcal{E}} %

\newcommand{\proj}{\Pi} %
\newcommand{\dsk}{\mathbb{D}}

\begin{document}

\begin{abstract}
Computational multi-scale methods capitalize on a large time-scale separation to efficiently simulate slow dynamics over long time intervals. For stochastic systems, one often aims at resolving the statistics of the slowest dynamics. This paper looks at the efficiency of a micro-macro acceleration method that couples short bursts of stochastic path simulation with extrapolation of spatial averages forward in time. To have explicit derivations, we elicit an amenable linear test equation containing multiple time scales. We make derivations and perform numerical experiments in the Gaussian setting, where only the evolution of mean and variance matters. The analysis shows that, for this test model, the stability threshold on the extrapolation step is largely independent of the time-scale separation. In consequence, the micro-macro acceleration method increases the admissible time steps far beyond those for which a direct time discretization becomes unstable.
\end{abstract}

\maketitle

\section{Introduction}\label{se:intro}

Models with multiple time scales abound in a large variety of domains: complex fluids, materials research, life sciences and bio-mechanics, to name a few~\cite{BriLel2009,Keunings2004,PraSitKre2008,DeiWanMacCri2011,FavKraPiv2016}.
At the same time, the design and analysis of efficient numerical methods for multi-scale stochastic differential equations (SDEs) remain challenging. While explicit schemes require excruciatingly small time steps, implicit schemes -- though successful for stiff ordinary differential equations (ODEs)~-- may yield an incorrect invariant distribution in the stochastic case~\cite{LiAbdE2008}. Thus, dedicated computational multiscale methods for SDEs are required. An extensive body of work already exists, see \cite{KevSam2009,GivKupStu2004,AbdEEngEij2012} and references therein.

In~\cite{DebSamZie2017}, we introduced a micro-macro acceleration method for the simulation of SDEs with a separation between the (fast) time scale of individual trajectories and the (slow) time scale of the macroscopic function of interest. The method couples the microscopic model, of which we have full knowledge, to a macroscopic level, described by a finite set of macroscopic state variables -- averages over the microscopic distribution. To bypass the prohibitive cost of the direct Monte Carlo simulation, the method alternates between short bursts of microscopic path simulation and extrapolation of macroscopic states forward in time.
One time step $\De t$ of the micro-macro acceleration method consists of: (i) microscopic \emph{simulation} of the full stochastic system with a~small batch $K$ of micro time steps $\de t$; (ii) \emph{restriction} of the simulated paths to $L$ macroscopic states at each of the $K+1$ time-points; (iii) forward in time \emph{extrapolation} of the macroscopic state over $\De t - K\de t$, based on the estimation of the macroscopic time derivative from the states obtained in (ii); and (iv) \emph{matching} of the last microscopic state from stage (i) with the extrapolated macroscopic state.
The matching is an inference procedure that renders a minimal perturbation of a prior microscopic state (available just before the extrapolation) consistent with the extrapolated macroscopic state.

In this manuscript, we study \emph{asymptotic numerical stability} of the micro-macro acceleration method -- preservation of asymptotic qualitative behaviour of equilibria under time discretisation\footnote{Not to be confused with \emph{numerical stability} (necessary for convergence) that measures the robustness of a numerical scheme with respect to perturbations, such as round-of errors, over finite time horizon as step-size tends to zero.} -- in terms of the extrapolation time step and the time-scale separation present in the system.
Linear stability analysis constitutes a necessary part of analysis for any numerical scheme.
Indeed, to be of computational interest, the micro-macro acceleration method should satisfy two basic properties: (i) it should \emph{converge} to the full microscopic dynamics in an appropriate limit; and (ii) it should be \emph{more efficient} than a full microscopic simulation.

In~\cite{DebSamZie2017,LelSamZie2019}, item (i) -- the convergence --  was rigorously analyzed in a general setting for non-linear SDEs. The results demonstrate that the micro-macro acceleration method recovers the exact evolution of the SDE in the limit when the number $L$ of macroscopic state variables tends to infinity and both the microscopic $\de t$ and macroscopic $\De t$ time steps tend to zero. The quality of approximation during convergence was measured by the weak error between the exact and numerical solutions or in $L^2$ norm between the corresponding laws. In~\cite{DebSamZie2017}, the rate of convergence in terms of $L$, $\de t$ and $\De t$ was obtained under generic assumptions on the matching, whereas~\cite{LelSamZie2019}  gives the proof of the convergence without a rate for matching based on minimization of Kullback-Leibler divergence.
The main results of this work tackle item (ii) for a simple linear test case and demonstrate that the micro-macro acceleration method increases the admissible time step for stiff systems well beyond step sizes for which a direct time discretization becomes unstable. In the models we analyse, the stiffness results from the presence of the fast and slow microscopic variables -- evolving on vastly different time scales quantified by the eigenvalues of the associated linear operator. At the macroscopic level, we only extrapolate averages of the slow variables. We show that in this case, the extrapolation time step is only limited by the ``slow'' eigenvalues of the system.
Therefore, the extrapolation time step is not artificially constrained by the ``fast'' eigenvalues, and can be determined purely based on accuracy considerations. (We defer a precise definition of ``slow'' and ``fast'' eigenvalues to Section~\ref{se:block_diag}.)

Before proceeding to the technical content of this manuscript, we briefly review the complications that arise when analyzing linear stability in the stochastic setting.

\textbf{Linear stability of numerical methods.}
In deterministic numerical analysis, the study of linear stability of time-discretization methods focuses on the simple scalar linear test equation $\dot{x}=\ka x$~\cite{Iserles2008}.
The results can be transferred to systems of linear ODEs by decoupling the system into a set of scalar (complex) test equations using spectral factorization. One then considers the same linear scalar equation, in which $\ka$ runs over all eigenvalues of the linear system.
To generalize the analysis to the asymptotic stability of nonlinear systems of ODEs, one linearizes around an equilibrium solution, leading to a~vector linear equation.
 The validity of this last step relies on the Hartman-Grobman or first approximation theorems, which bring the linear stability results back to the nonlinear problem~\cite{Iserles2008,Teschl2012}.

Obtaining similar stability results for SDEs is highly non-trivial for at least three reasons.
First, the transfer of linear stability results of numerical methods to the nonlinear setting remains unjustified in the case of stochastic systems~\cite{BucRieKlo2011}. The link between linear and nonlinear SDEs is intricate and not yet understood in full generality. Some interesting results on this point, based on normal form transforms, have been reported in~\cite{Roberts2008a,Roberts2014}. We do not address this issue here and instead start from a linear test equation from the outset. Specifically, we study linear vector stochastic differential equations in $\R^d$ with additive noise
\begin{equation}\label{eq:linSDE}
\der{X}_t = AX_t\der{t} + \sqrt{B}\der{W}_t,
\end{equation}
in which $W_t$ is a standard $d$-dimensional Wiener process on $[0,\infty)$ and the constant, nonrandom matrices $A\in\mat{d}{d}$ and $B\in\mat{d}{d}$ are called the drift and diffusion matrix, respectively.

Second, the connection between the vector and scalar cases is not a~straightforward extension of the deterministic case.
 In general, this connection can only be achieved under the~assumption of simultaneous diagonalizability of the matrices appearing in~\eqref{eq:linSDE}~\cite{BucKel2010}.
  In the context of slow-fast SDEs, this assumption is too restrictive. %
   Thus, an important issue we address in this paper concerns the definition of a suitable slow-fast vector test equation that allows studying linear stability of the micro-macro acceleration method.
Finally, both the choice of test equation and stability concept vary in the SDE literature~\cite{BucRieKlo2011}.
 Concerning the linear test equation, we face the alternative of SDEs with multiplicative or additive noise, which display different qualitative behaviour but both arise from a perturbation of the same underlying ODE.
 Regarding the stability concept in the stochastic setting, we distinguish between strong stability, which usually looks at the almost sure or mean square convergence of paths of the process to a suitable stationary solution, and weak stability, which investigates the convergence (e.g., in distribution) of laws of the process to an invariant measure of the SDE.
  Until now, research mostly focused on (both linear and non-linear) SDEs with multiplicative noise and strong stability of zero solutions of the discrete and continuous systems~\cite{Higham2000a,Saito2002,DebRos2009,BucKel2010,AbdBlu2013,DiaJer2016,SzpZha2018}.
  Also, several works investigating the weak stability of such systems exist~\cite{YuaMao2004,LiuMao2015}.
   For the additive noise, the definition of the stationary solution requires the theory of random dynamical systems~\cite{CruBisJimCarOza2010,BucRieKlo2011,DiaJer2016}.

In our study, we focus on the convergence of the (time-marginal) laws of~\eqref{eq:linSDE} to its invariant distribution.
 We therefore choose the combination of additive noise and weak stability, incidentally the least represented in the literature so far, see~\cite{HerSpi1992,Saito2008}.
 We consider the perturbation by additive noise in~\eqref{eq:linSDE} since it leads to a non-degenerate Gaussian stationary distribution, as opposed to a Dirac mass in the multiplicative case.
 Moreover, the nature of the micro-macro acceleration method -- which, by the presence of forward extrapolation of averages and the matching procedure, produces only weak approximations to the underlying SDE -- motivates the choice of the weak stability concept.

\textbf{Organization of the paper.}
In Section~\ref{se:mM_accel}, we describe the micro-macro acceleration method for general SDEs and present two matching procedures based on minimization of Kullback-Leibler divergence: one with restriction to the mean only, and a second one with both mean and variance.
In Section~\ref{se:LinSDEs_Stab}, we introduce our model problems -- linear vector SDEs with time scale separation -- and give the corresponding decomposition into slow and fast variables.
Section~\ref{se:mM_accel_linSDE} merges the two previous ones by specifying the micro-macro acceleration method for the study of linear stability of slow-fast SDEs.
Finally, in Section~\ref{se:mM_stab_gauss}, we investigate the linear stability of the micro-macro acceleration method for Gaussian initial conditions under the assumption that exact distribution is propagated. For some analytical results in the non-Gaussian case we refer to~\cite{Zielinski2019}. %
Here, we also complement the theoretical findings with numerical experiments performed in the Gaussian setting, but based on approximating distributions by empirical measures, and show accordance with analytical bounds. Some numerical results for non-linear systems can be found in~\cite{Vandecasteele2019}.

\textbf{Basic notation.}
\textit{Probability.}
By $\prob^d$ we denote the collection of all probability measures on $\R^d$ that act on the $d$-dimensional Borel $\si$-algebra $\bor^d$.
In particular, $\nd_{\mu,\Si}\in\prob^d$ stands for the normal distribution with vector mean $\mu$ and covariance matrix $\Si$.
For any $P\in\prob^d$, $\Exp[P][\,\cdot\,]$ and $\Var[\!P][\,\cdot\,]$ are the mean and covariance operators of $P$, respectively.
 If $X$ is a~random vector distributed according to $P$, in short $X\sim P$, we indicate by $\Exp[][X]$ its vector mean and by $\Var[][X]$ its covariance matrix.

\textit{Matrices.}
For $M\in\mat{d}{d'}$, the space of $d\times d'$ real matrices, $\tp{M}$ is the transpose of $M$, and if $d=d'$, $\spm{M}$ stands for the spectrum and $\tr{M}$ for the trace of $M$.
In what follows, the matrix product of $M$ and $M'$ is $MM'$, and for $M,M'\in\mat{d}{d'}$, we denote by $M\sdot M'\doteq\tr{\tp{M}M'}$ their Frobenius inner product.
In particular, if we have two (column) vectors $x,y\in\R^d=\mat{d}{1}$, $x\sdot y=\tp{x}y$ equals their scalar product.
We also use the symbol $\Id$ to denote the identity matrix in $\mat{d}{d}$.

\textit{Direct sums.}
For $d=d_s+d_f$, we will consider a~decomposition of $\R^d$ into the (direct) sum $\R^d=\R^{d_s}\oplus\R^{d_f}$. Accordingly to this decomposition, for $y\in\R^{d_s}$ and $z\in\R^{d_f}$, we denote $x=y\oplus z\in\R^d$ the vector concatenation of $y$ and $z$. In the remainder of the manuscript, we always place the indices related to the decomposition of $\R^d$ in the superscripts of relevant symbols, for example $\prob^s$ indicates the set of all probability measures on $\R^{d_s}$, reserving subscripts to indicate time and/or step number.

\textbf{Code.}
The source  code, in Python 3, can be found at \cite{source_code}.

\section{Micro-macro acceleration with Kullback-Leibler divergence matching}\label{se:mM_accel}
In this section, we describe the micro-macro acceleration method, which was first introduced in~\cite{DebSamZie2017},
as it applies to the simulation of a general It\^{o} SDE
\begin{equation}\label{eq:genSDE}
\der{X}_t = a(X_t)\der{t} + b(X_t)\der{W}_t.
\end{equation}
The basic scenario under which the method could be applied reads:
we are interested in the long time simulation of some observables $\Exp[][f(X_t)]$ of~\eqref{eq:genSDE} while the time scale on which these observables evolve is much slower than the one resolving the fine dynamics of $X_t$.

\begin{example}[Slow-fast systems]\label{ex:sfSDE}
A simple illustration of a process of interest is $X_t=(Y_t,Z_t)$ that solves the equation
\begin{gather}\label{eq:sfSDE}
\begin{aligned}
\der{Y}_t &= \ \;a_{1}(Y_t ,Z_t)\,\der{t} &+\ \hphantom{\vep^{-1/2}}\der{U}_t,\\[0.2em]
\der{Z}_t &= \vep^{-1}a_{2}(Z_t)\,\der{t} &+\ \vep^{-1/2}\der{V}_t,
\end{aligned}
\end{gather}
with $\vep>0$. For small $\vep$, the \emph{fast variable} $Z_t$ requires a fine numerical procedure to resolve. However, we could be only interested in the long time values of the moments of \emph{slow variable} $Y_t$, such as $\Exp[][Y_t]$ or $\Var[][Y_t]$.
\end{example}

The method is designed to benefit from the focus on slow moments only and bridge the scale separation by ``forgetting'' the fast variables repeatedly on a considerable part of the time domain. The stability analysis we perform in this manuscript, on the linear version of~\eqref{eq:sfSDE} (see Example~\ref{ex:test_sf}), aims at showing that the stability of the micro-macro acceleration method is independent on the scale separation $\vep$. Therefore, in such cases there are no artificial stability constraints on the macro time step due to fast variables.

\begin{remark}[On theoretical vs numerical setting]
In the next section, we present the algorithm as a procedure that propagates the exact laws of random variables over a time mesh; this is the framework we analyse in this manuscript. For numerical experiments, we further discretize the probability distributions using the Monte Carlo procedure. A comprehensive discussion on the implementation issues when using empirical distributions as approximation to exact laws can be found in~\cite{DebSamZie2017}. For convenience, we also detail how the method works with path simulations in Appendix~\ref{se:MC_sim}.
\end{remark}

\subsection{Micro-macro time acceleration.}\label{se:mM_accel_alg}
Given a distribution $P^{\De t}_n\in\prob^d$ at time $t_n=n\De t$, one time step $\De t$ of the micro-macro acceleration method advances $P^{\De t}_n$ to a distribution $P^{\De t}_{n+1}$, approximating the law of solution to SDE~\eqref{eq:genSDE} at time $t_n+\De t$, in four stages.

\textit{Stage 1: Propagation of microscopic laws.}
This stage is based on an \emph{inner integrator}: a numerical scheme $\prop^{\de t}$ for equation~\eqref{eq:genSDE}, where $\de t$ is a micro time step. For a natural number $K$ such that $K\de t<\De t$, we generate $K$ probability distributions as follows
\begin{equation}%
P^{\de t}_{n,k+1} = \Law{X^{\de t}_{n,k+1}}, \quad X_{n,k+1}^{\de t} = \prop^{\de t}(a, b, X_{n,k}, \de W_{k+1}),\quad X_{n,0}^{\de t}\sim P^{\De t}_{n,0}:=P^{\De t}_n,
\end{equation}
where $\de W_{k+1} = W_{(k+1)\de t} - W_{k\de t}$ are Brownian increments.

\textit{Stage 2: Restriction to a~finite number of observables.}
For this part, we introduce a restriction operator $\res\from\dom{\res}\subseteq\prob^d\to\R^L$ that associates to a probability distribution a number $L$ of its averages. Using $\res$, we evaluate the corresponding macroscopic states for the distributions obtained in stage $1$ by setting
\begin{equation*}
\mbf{m}_{n,k} \doteq \res(P^{\de t}_{n,k}),\quad k=0,\dotsc,K.
\end{equation*}

\textit{Stage 3: Extrapolation of macroscopic states.}
We use macroscopic states $\{\mbf{m}_{n,k}\}_{k=0,\dotsc,K}$, at times $t_n+k\de t$, and an extrapolation operator $\ext^{\De t-K\de t}$ over $\De t-K\de t$ to compute the extrapolated macroscopic state $\mbf{m}_{n+1}$ at time $t_{n+1}=t_n+\De t$, i.e.
\begin{equation}\label{eq:extrap}
\mbf{m}_{n+1} = \ext^{\De t-K\de t}(\mbf{m}_{n,0},\dotsc,\mbf{m}_{n,K}).
\end{equation}

\textit{Stage 4: Matching.}
In the last stage, we infer a~new distribution on $\R^d$, compatible with the extrapolated slow moments $\mbf{m}^s_{n+1}$. For this, we use the matching operator $\match$ (see Section~\ref{se:match}) and put
\begin{equation*}
P^{\De t}_{n+1} = \match(\mbf{m}_{n+1},P^{\de t}_{n,K}).
\end{equation*}

Matching is a crucial stage of the algorithm.
In this stage, we aim to obtain, after extrapolation, a new probability distribution of the full system compatible with given macroscopic states. To address and regularize this inference problem, our strategy for matching uses a~\emph{prior} distribution $P^{\de t}_{n,K}\in\prob^d$, which comes from the last available microscopic state at Stage 1, and alters it to make it consistent with the~extrapolated macroscopic state~$\mbf{m}_{n+1}$. We give two examples in Section~\ref{se:match}.

The other components of the micro-macro acceleration method vary and are specific to a given problem one wants to address, see~\cite{DebSamZie2017} and Section~\ref{se:mM_accel_alg_linSDE}. The scheme for the inner integrator $\prop$ can be any method suitable to simulate~\eqref{eq:genSDE} with desired accuracy and within given constraints such as: the explicit Euler-Maruyama (which we use in this paper), predictor-corrector (when the range of $X$ is bounded), or higher order schemes (to achieve better convergence). The restriction operator $\res$ corresponds to our choice of macroscopic description of system~\eqref{eq:genSDE} and is most often given by averaging a vector of functions $\mbf{R}$, so that $\res(P^{\de t}_{n,k}) = \Exp[][\mbf{R}(X^{\de t}_{n,k})]$. The extrapolation operator usually employs the  difference quotients to approximate the derivative $\dot{\mbf{m}}(t)$ of the evolution of macroscopic variables $\mbf{m}(t)=\res(X_t)$ and uses this approximation to extrapolate the macroscopic state forward in time.

\begin{remark}[On the number of micro steps]
In general, there are two main reasons to consider $K>1$ in Steps~2 and~3: when we need to equilibrate the fast variables conditioned on the slow ones before performing extrapolation, or when using higher order extrapolation that requires macroscopic observables at multiple time points. Both of these reasons are related to accuracy issues, rather than stability, and are more relevant in the non-linear settings. Thus, we restrict the theoretical analysis and numerical experiments to the case of $K=1$ and linear extrapolation (see Section~\ref{se:mM_accel_alg_linSDE}), save one numerical illustration in Section~\ref{se:mM_gauss_smenvar_num}.
\end{remark}

\subsection{Matching with mean and variance via minimum Kullback-Leibler divergence}\label{se:match}
In this manuscript, we build upon \cite{DebSamZie2017,LelSamZie2019} and obtain the matched distribution from the prior $P$ by minimising the \emph{Kullback-Leibler divergence} (also called \emph{logarithmic relative entropy})
\begin{equation}\label{eq:kld}
\kld(Q\|P)\doteq\Exp[Q]\Big[\ln\!\rnder{Q}{P}\,\Big]
\end{equation}
over all probability distributions $Q$ absolutely continuous with respect to $P$ and having correct (extrapolated) macroscopic state.
The Kullback-Leibler divergence between two probability distributions is always non-negative (a~consequence of Jensen's inequality), but it is not a~metric on $\prob^d$: it lacks the symmetry property and fails to satisfy the triangle inequality.
Nevertheless, we can use $\kld$ to quantify the proximity of two distributions since it provides an upper bound on the \emph{total variation distance} $d_{\mrm{TV}}(Q,P)\doteq\sup\{|Q(U)-P(U)|:\ U\in\bor^d\}$ via Pinsker's inequality:
$d_{\mrm{TV}}(Q,P)\leq\sqrt{\kld(Q\|P)}$~\cite{Pinsker1964}.
For additional motivation for the use of the Kullback-Leibler divergence for matching see~\cite{LelSamZie2019}.

We are specifically interested in two particular inference procedures to recover the microscopic distribution: one based on the mean only, and one based on both mean and variance. For a general description of matching with arbitrary moments, we refer to~\cite{DebSamZie2017}.

\textbf{Matching with mean only.} The first procedure, preserving the mean only, reads
\begin{equation}\label{eq:match_mean}\tag{ME}
\match(\oline{\mu},P) = \argmin_{Q\in\prob^d}\kld(Q\|P),\ \text{constrained on}\ \Exp[Q][\Pi]=\oline{\mu},
\end{equation}
where $\Pi$ denotes the identity operator\footnote{We use $\Pi$ to unify the notation between this section and the subsequent one, in which we will project onto a lower-dimensional subspace. Here, one can think of $\Pi$ as the projection from $\R^d$ onto $\R^d$.} in $\R^d$. The density of the matched distribution $Q^{\oline{\mu}}=\match(\oline{\mu},P)$ with respect to the prior $P$ is~\cite{DebSamZie2017}
\begin{equation}\label{eq:dist_mean}
\rnder{Q^{\oline{\mu}}}{P}(x)
= \exp\!\big(\oline{\la}\sdot x - A(\oline{\la},P)\big),
\end{equation}
where the log-partition function $A(\la, P) =
\ln\Exp[P]\big[\exp(\la\sdot\Pi)\big]$ gives the normalizing constant in~\eqref{eq:dist_mean} and $\oline{\la}\in \R^d$ is a~vector of Lagrange multipliers that satisfies
\begin{equation}\label{eq:lagr_mean}
\grad[\la]A(\oline{\la},P) = \oline{\mu}.
\end{equation}
Note that the log-partition function $A(\la,P)$ for matching with mean only coincides with the value of the \emph{cumulant generating function} (CGF) of $P$ at $\la$, see~\cite[A.1]{JorLab2012}.
Let us point out that $A(\la,P)$ is finite only when $\la$ belongs to the domain of the CGF of $P$, and we can look for solutions to~\eqref{eq:lagr_mean} only in this domain. We do not consider, in its full generality, the question whether Lagrange multipliers exist for any given prior $P$ and mean $\mu$. In Section~\ref{se:mM_stab_gauss}, we work within a Gaussian framework, where we can solve~\eqref{eq:lagr_mean} explicitly.

\textbf{Matching with mean and variance.} The second procedure employs additionally the full covariance matrix and it reads
\begin{equation}\label{eq:match_meanvar}\tag{MEV}
\match(\oline{\mu},\oline{\Si},P) = \argmin_{Q\in\prob^d}\kld(Q\|P),\ \text{constrained on}\ \left\{\begin{array}{l}
\Exp[Q][\Pi]=\oline{\mu}\\
\Var[Q][\Pi] = \oline{\Si}.
\end{array}\right.
\end{equation}
We introduce the vectors of Lagrange multipliers $\oline{\la}\in\R^d$ and $\oline{\La}\in\R^{d\times d}$, and write the Radon-Nikodym derivative of the matched distribution $Q^{\oline{\mu},\oline{\Si}} = \match(\oline{\mu},\oline{\Si},P)$ as~\cite{DebSamZie2017}
\begin{equation*}
\rnder{Q^{\oline{\mu},\oline{\Si}}}{P}(x) = \exp\!\big(\oline{\la}\sdot x + \oline{\La}\sdot(x-\Exp[P][\Pi])^2 - A(\oline{\la},\oline{\La},P)\big),
\end{equation*}
where
$v^2\doteq v\tp{v}$ for any $v\in\R^d$. In this case, the log-partition function reads $A(\la,\La,P) = \ln\Exp[P]\big[\exp\big(\la\sdot\Pi + \La\sdot(\Pi-\Exp[P][\Pi])^2\big)\big]$ and the Lagrange multipliers $(\oline{\la},\oline{\La})$ satisfy
\begin{equation}\label{eq:lagr_meanvar}
\left\{\!\begin{array}{l}
\grad[\la]A(\oline{\la},\oline{\La},P) = \oline{\mu}\\
\grad[\La]A(\oline{\la},\oline{\La},P) = \oline{\Si}.
\end{array}\right.
\end{equation}
Note that, due to symmetry in $\oline{\Si}$, the systems for the constraints and for the Lagrange multipliers are overdetermined. In fact, to obtain the Lagrange multipliers, we only need to solve $d(d+3)/2$ non-linear equations. Similarly as for the matching with mean only, there arises a question whether we can find a solution to the above system for a given prior $P$ and $(\oline{\mu},\oline{\Si})$. In our analysis and during numerical experiments, we employ~\eqref{eq:match_meanvar} only with Gaussian priors, thus we can always find a solution provided the matrix $\oline{\Si}$ is non-negative definite.

\section{Model problems: linear slow-fast SDEs\label{se:LinSDEs_Stab}}
Having described the micro-macro acceleration algorithm, we now specify the model systems on which we investigate the stability of the method. We focus on a~particular subclass of linear SDEs with a large gap in the spectrum of their drift matrix that yields a time scale separation between certain variables of the equation. This property is exemplified in the following system that can be seen as a ``linearization'' of~\eqref{eq:sfSDE} from Example~\ref{ex:sfSDE}.
\begin{example}[Linear slow-fast systems]\label{ex:test_sf}
Consider simulation of $X_t=(Y_t,Z_t)$ given by
\begin{gather}\label{eq:test_sf}
\begin{aligned}
\der{Y}_t &= (a_{11}Y_t + a_{12}Z_t)\der{t} &+\ \hphantom{\vep^{-1/2}}\der{U}_t,\\[0.2em]
\der{Z}_t &= \hphantom{a_{21}Y_t\ } \vep^{-1}a_{22}Z_t\,\der{t} &+\ \vep^{-1/2}\der{V}_t,
\end{aligned}
\end{gather}
with $a_{11},a_{12},a_{22}\in\R$ and $\vep>0$. Here the drift and diffusion matrices read
\begin{equation*}
A =
\begin{bmatrix}
a_{11} & a_{12}\\
0 & a_{22}/\vep
\end{bmatrix},\quad B =
\begin{bmatrix}
1 & 0\\
0 & 1/\vep
\end{bmatrix}.
\end{equation*}
The spectrum $\spm{A} = \{a_{11},a_{22}/\vep\}$ resides in the left-half complex plane if $a_{11},a_{22}<0$. As $\vep$ tends to $0$, when the system becomes \emph{slow-fast}, we see that $|a_{22}/\vep|\gg |a_{11}|$ and $\spm{A}$ splits into two distant parts. This rupture does not interfere with the stability of the exact solution to equation~\eqref{eq:test_sf} but, as described in Section~\ref{se:mM_invdistr}, it has profound consequence on the stability of the numerical schemes as the system becomes stiff.
\end{example}

In Section~\ref{se:spec-drift}, we generalize the qualitative behavior of~\eqref{eq:test_sf} to a linear system~\eqref{eq:test} below. Equation~\eqref{eq:test} does not generally decouple into a number of independent complex SDEs, as a simultaneous dizaonalization of the matrices $A$ and $B$ in equation~\eqref{eq:test} is not normally possible for multiscale systems. In Section~\ref{se:block_diag}, we circumvent this problem by only performing a block diagonalization of the drift matrix; a procedure that is always possible to adopt for equation~\eqref{eq:test}.

To introduce the model problem, let us consider the standard $d$-dimensional Wiener process $W_t$ on $[0,\infty)$ and constant, nonrandom matrices $A\in\mat{d}{d}$ and $B\in\mat{d}{d}$, which we call the drift and diffusion matrix, respectively. We make a~standing assumption that the initial random vector $X_{0}$ is independent of $W$ and that $B$ is symmetric positive definite. To make the formulas in later sections more concise, we consider the following linear vector stochastic differential equation in $\R^d$, see also Remark~\ref{re:lin},
\begin{equation}\label{eq:test}
\der{X}_t = AX_t\der{t} + \sqrt{B}\der{W}_t,\quad X_{0}\ \text{given},
\end{equation}
where $\sqrt{B}$ is the symmetric positive definite square root of $B$~\cite[Thm.~7.2.6]{HorJoh2013}. The positive definiteness of the diffusion matrix $B$ ensures that the distribution of $X_t$ stays non-degenerate whenever this is the case for $X_0$. Considering the goal of this manuscript, covering degenerate diffusions does not add any substantial insight to the stability analysis, whereas being able to use densities makes derivations more lucid.

\begin{remark}[From linear systems in the narrow sense to~\eqref{eq:test}]\label{re:lin}
In~\eqref{eq:test}, we fix the dispersion coefficient to the square root of a $d\times d$ diffusion matrix $B$. In the context of linear SDEs with constant coefficients, it is customary to consider equation (called linear in the narrow sense~\cite[p.~110]{KloPla1999})
\begin{equation}\label{eq:lin}
\der{X}_t = AX_t\der{t} + \oline{B}\der{\oline{W}}_t,
\end{equation}
where $\oline{B}\in\mat{d}{d'}$ and $\oline{W}_t$ is a $d'$-dimensional Brownian motion. To preserve nondegeneracy, it suffices to assume\footnote{More generally, the sufficient and necessary condition requires that the pair $(A,\oline{B})$ is \emph{controllable}, see~\cite[pp.~355-356]{KarShr1998}. Assuming controllability only, we may not always be able to reduce~\eqref{eq:lin} to~\eqref{eq:test}. However, as~\eqref{eq:test} serves as a convenient test equation for asymptotic stability, this discrepancy is of minor importance.} that the rank of $\oline{B}$ equals $d$, yielding $d\leq d'$.
When this is the case, the matrix $B=\oline{B}\,\tp{\oline{B}}$ is symmetric positive definite~\cite[p.~440]{HorJoh2013}, $\oline{C}=(\sqrt{B})^{-1}\oline{B}\in\mat{d}{d'}$ has orthogonal rows,
and we can decompose $\oline{B}=\sqrt{B}\,\oline{C}$. Due to orthogonality, the process $W_t=\oline{C}\,\oline{W}_t$ is a $d$-dimensional Wiener process, and plugging the decomposition of $\oline{B}$ into~\eqref{eq:lin} brings us back to~\eqref{eq:test}.
\end{remark}

Let us shortly recall asymptotic stability of the exact trajectories of~\eqref{eq:test}, see~\cite{KarShr1998}. Provided that $X_0$ has the first two moments finite and under the assumption that all the eigenvalues of $A$ have negative real parts, the integral
\begin{equation}\label{eq:test_equivar}
V_\infty \doteq \int_0^\infty e^{uA}Be^{u\tp{A}}\der{u}
\end{equation}
converges, is positive definite, and it holds that
\begin{equation}\label{eq:test_meanvar_asym}
\lim_{t\to\infty}\Exp{[X_t]} = 0,\quad \lim_{t\to\infty}\Var{[X_t]} = V_\infty.
\end{equation}
When $X_0$ is a~Gaussian vector, by linearity, the solution $X_t$ of~\eqref{eq:test} stays Gaussian as well. In consequence, employing the formula for the Kullback-Leibler divergence between two Gaussian vectors~\cite[p.~189]{Kullback1978} and using~\eqref{eq:test_meanvar_asym} to pass to the limit as $t$ goes to $\infty$, we~obtain
\begin{equation*}
\kld(\Law{X_t}\|\nd_{0,V_\infty}) = \frac{1}{2} \Big[\tr{V_\infty^{-1}\Var{[X_t]}} - d + \tp{\Exp{[X_t]}}V_\infty^{-1}\Exp{[X_t]} + \ln\frac{\|V_\infty\|}{\|\Var{[X_t]}\|}\Big]\ \convto\ 0.
\end{equation*}
Thus the distribution of $X_t$ converges to $\nd_{0,V_\infty}$, the invariant distribution of~\eqref{eq:test}, in the Kullback-Leibler divergence.

For the linear stability analysis of the micro-macro acceleration method, we are concerned with the asymptotics of numerical trajectories in $\prob^d$ of equation~\eqref{eq:test} under the micro-macro acceleration method discussed in Section~\ref{se:mM_accel}. Here and in what follows, we look at the asymptotic stability as a~property of the laws $P^{\De t}_n$ of time marginals; we require that these laws converge (with an appropriate notion of convergence for probability measures) to some fixed probability distribution $P_{\infty}$ as time goes to infinity: reminiscent of the convergence of $\Law{X_t}$ to  $\nd_{0,V_\infty}$. In Section~\ref{se:mM_invdistr}, we find $P_{\infty}$ as an invariant measure of the underlying inner integrator: the Euler-Maruyama scheme.

\subsection{Spectral properties of the drift matrix and slow-fast decomposition\label{se:spec-drift}}
We consider a~particular subclass of linear SDEs~\eqref{eq:test} for which the spectrum of the drift matrix $A$, denoted $\spm{A}$, lies in the (open) left-half complex plane $\C_{-}$ and decomposes into two pieces with a large gap between the moduli of eigenvalues. More precisely, we posit:
\begin{assumption}\label{as:test_mscale}
There exist two disjoint subsets $\Om^s,\Om^f$ of the left-half complex plane $\C_{-}$ such that $\spm{A}=\Om^s\cup\Om^f$ and
\begin{equation}\label{eq:test_mscale}
\max\{|\ka|:\ \ka\in\spm{A}\cap\Om^s\}\ll\min\{|\ka|:\ \ka\in\spm{A}\cap\Om^f\},
\end{equation}
where by ``$\ll$'' we mean ``much smaller'' and leave the quantification of this statement open to interpretation contingent on a particular case at hand.
\end{assumption}

In what follows, we call elements of $\Om^s$ the \emph{slow eigenvalues} of equation~\eqref{eq:test} and denote their cardinality by $d_s$; likewise, the $d_f$ \emph{fast eigenvalues} are the members of $\Om^f$. This terminology relates to the time scales present in system~\eqref{eq:test}. To see this, consider for a moment the deterministic part of~\eqref{eq:test}: $\dot{x}=Ax$. The solution $x(t)=e^{tA}x(0)$ decays to zero as $t$ goes to infinity, since all eigenvalues have negative real part. We can discern $d$ time scales in this decay that are quantified by their exponential rate of decay, proportional to $\re(\ka)$, and the frequency of oscillations, proportional to $\im(\ka)$, where $\ka$ runs through all eigenvalues of $A$. Therefore, the condition~\eqref{eq:test_mscale} imposes two (groups of) vastly different time scales of $x(t)$, corresponding to the gap in $\spm{A}$: \emph{slow}, for which both $\re(\ka)$ and $\im(\ka)$ are small, and \emph{fast}, for which either $\re(\ka)$ or $\im(\ka)$ is large. We can expect these different time scales  to prevail in the stochastic system, since~\eqref{eq:test} results from agitating the deterministic part with additive noise. We discuss in more detail the time scales present in the stochastic case in Appendix~\ref{se:tscale_dynsys}.

With the decomposition of $\spm{A}$ comes naturally the representation of the state space $\R^d$ as the direct sum of a $d_s$- and a $d_f$-dimensional subspaces, the invariant subspaces associated to the slow and fast eigenvalues, respectively.
Indeed, by the Spectral Decomposition Theorem, there exist orthogonal projections $\Pi^s,\Pi^f$ from $\R^d$ to $d_s$- and $d_f$-dimensional subspaces of $\R^d$, respectively, such that
\begin{equation}\label{eq:spm_dec}
\R^d = \Pi^s\R^d \oplus \Pi^f\R^d,\quad\text{and}\quad
A = D^s\Pi^s \oplus D^f\Pi^f,
\end{equation}
where $D^s$ and $D^f$ are block-diagonal, square matrices of size $d_s$ and $d_f$ (see also Section~\ref{se:block_diag}). Accordingly, for any $x\in\R^{d}$ we write $x=y\oplus z\in\R^d$ where $y=\Pi^sx\in\R^{d_s}$ denotes the $d_s$-dimensional vector of \emph{slow variables} and $z=\Pi^fx\in\R^{d_f}$ the $d_f$-dimensional vector of \emph{fast variables}.
\begin{remark}[On the connection with coarse-graining]
Decomposition~\eqref{eq:spm_dec} is within the framework of coarse-graining, where the coarse-graining operator is given by the projection $\Pi^s$ onto the slow variables. In this approach, one aims to approximate the evolution of the projected process $\Pi^sX_t$ by building a~closed Markov dynamics on $\Pi^s\R^d$~\cite[Sect.~3]{HarKalKatPle2016}. In our case, we look only at the moments of the projected process $\Pi^sX_t$ and the micro-macro acceleration method allows us to build the closure ``on the fly''.
\end{remark}

Slow-fast systems like equation~\eqref{eq:test_sf} constitute an example of models with vast time-scale differences \cite{GivKupStu2004}. The feature~\eqref{eq:test_mscale} appears generally in \emph{multiscale systems} where one couples macroscopic evolution equations to microscopic ones to increase the accuracy of description of a particular model~\cite{BlaBriLegLel2010}. From the computational perspective, this coupling produces \emph{stiff} system of equations. According to this interpretation, equation~\eqref{eq:test} together with condition~\eqref{eq:test_mscale} constitutes a~simple model for such a~situation and we test the micro-macro acceleration method against its asymptotic behaviour.

Even though the effects of microscopic variables play an important role in obtaining sufficient accuracy of the simulation, the computational interest of the multiscale systems often lies in the macroscopic observables, i.e., the averages over the slow variables. The evolution of fast and slow variables proceeds on very different time scales, which makes the simulation a~difficult task. To overcome this issue, the micro-macro acceleration method of Section~\ref{se:mM_accel_alg_linSDE} aims at breaking the time-step barrier for the slow variables.

\subsection{Imposing block-diagonal structure on the drift}\label{se:block_diag}
As we mentioned in the introduction, no connection between the scalar SDEs with additive noise and the linear stochastic system~\eqref{eq:test} generally exists. Only when matrices $A$ and $B$ are \emph{simultaneously dizaonalizable} can we decouple~\eqref{eq:test} into a~number of independent complex SDEs with additive noise. However, simultaneous dizaonalization of drift and diffusion matrices never occurs in the context of multiscale models, compare with Example~\ref{ex:test_sf}, and we cannot use this assumption (nevertheless, see~\cite{KomMit1995} for the use of simultaneous diagonalizability in the case of multiplicative noise). Here, we only perform a block dizaonalization of the drift matrix instead; a~procedure always possible to adopt for equation~\eqref{eq:test}. This procedure simplifies derivations, allowing us to obtain certain analytical results about the asymptotics of the micro-macro acceleration method.

By Assumption~\ref{as:test_mscale}, we can find a~non-singular matrix $C\in\mat{d}{d}$, such that
\begin{equation*}
A = C^{-1}DC,
\end{equation*}
where $D=\diag(D^s, D^f)$ and $D^s\in\mat{d_s}{d_s}$, $D^f\in\mat{d_f}{d_f}$, $d_s+d_f=d$. The existence of $C$ is based on the \emph{real Jordan canonical form}, which gives the block diagonalization into a~number of full real Jordan blocks corresponding to either one real eigenvalue, or a~pair of complex eigenvalues. By rearranging, we can always combine multiple Jordan blocks into a~bigger block. Here, we align the blocks $D^s$ and $D^f$ along the decomposition of the spectrum $\spm{A}$ into the slow modes in $\Om^s$ and the fast ones in $\Om^f$, respectively.

By the It\^{o} formula, the stochastic derivative of the process $\otilde{X}_t\doteq CX_t$, where $X_t$ solves~\eqref{eq:test}, satisfies $\der{\otilde{X}}_t=CAX_t\der{t}+C\sqrt{B}\der{W}_t$. As a~result, $\otilde{X}_t$ is the solution to the following linear SDE
\begin{equation}\label{eq:bdiag}
\der{\otilde{X}}_t = D\otilde{X}_t\der{t}+\sqrt{\otilde{B}}\der{\otilde{W}}_t,\quad\otilde{X}_0=CX_0,
\end{equation}
where $\otilde{B}\doteq \tp{(C\sqrt{B})}C\sqrt{B}\in\mat{d}{d}$ (consider $\oline{B}=C\sqrt{B}$ in Remark~\ref{re:lin}) and $\otilde{W}_t$ is an orthogonal transformation of $W_t$.
The main advantage of the SDE system~\eqref{eq:bdiag}, with block diagonal drift matrix, is that it allows an analytical stability analysis.

Note that our standing assumptions still hold for~\eqref{eq:bdiag} -- the random variable $CX_0$ keeps being independent of the Wiener process $\otilde{W}_t$ and, since $C\sqrt{B}$ has full rank, the matrix $\otilde{B}$ is positive definite~\cite[p.~440]{HorJoh2013}. Moreover, the moments of $X_t$ and $\otilde{X}_t$ are connected through direct formulas involving the matrix $C$. In particular, we have $\Exp[][\otilde{X}_t]=C\Exp[][X_t]$ and $\Var[\otilde{X}_t] = C\Var[X_t]\tp{C}$, so that the asymptotic relations~\eqref{eq:test_meanvar_asym} hold for~\eqref{eq:bdiag} with limiting variance $CV_\infty\tp{C}$. Most importantly, however, the drift matrices of~\eqref{eq:test} and~\eqref{eq:bdiag} have the same eigenvalues, thus the asymptotic stability of both SDE systems is equivalent.

\begin{remark}[On issues with diagonaliation in the complex field]
Employing complex linear transformation $C\in\mat[C]{d}{d}$, we could dizaonalize the drift matrix $A$ whenever it had a~full set of eigenvectors. This procedure would tend to resemble the corresponding approach for linear stability of ODEs, which we discussed at the beginning of Section~\ref{se:mM_accel_linSDE}. However, applying complex transformations to the drift matrix of~\eqref{eq:test} yields a~number of difficulties in the interpretation of the resulting complex linear SDE systems.

First, we cannot keep equation~\eqref{eq:bdiag} as our test model, since, with complex $C$, the new drift matrix $D$ and diffusion matrix $\otilde{B}$ are in general complex too, while the Brownian motion $W_t$ stays real. Therefore, equation~\eqref{eq:bdiag} is not a~complex linear SDE in this case.
To transform~\eqref{eq:bdiag} into a complex linear SDE, we could define $\ohat{B}=\tp{(C\sqrt{B}C^{-1})}C\sqrt{B}C^{-1}$, and consider the complex process $\ohat{W}_t=C\otilde{W}_t$, so that the transformed equation reads
\begin{equation*}
\der{\ohat{X}}_t = D\ohat{X}_t\der{t} + \sqrt{\ohat{B}}\der{\ohat{W}}_t.
\end{equation*}
The problem is that $\ohat{W}_t$ is not a~Brownian motion, unless $C$ is orthogonal. It is possible to dizaonalize $A$ with an orthogonal matrix only if $A$ is symmetric in the first place. However, multiscale systems lack such symmetry property basically by their definition, see~\eqref{eq:test_sf} for an example.
More generally, also connecting the stability of real and complex linear equations remains problematic. The invariant distribution of a~complex equation is the complex normal distribution whose density has the form
\begin{equation*}
\frac{1}{\pi^d\det H}\exp\!\big(\!-\htp{\xi}H^{-1}\xi\big),\quad \xi\in\C^d,
\end{equation*}
where $\htp{\xi}$ is the~complex (Hermitian) conjugate and $H$ is a~complex positive-definite matrix (see~\cite[Chapter~2]{AndHojSorEri1995} for more on the multivariate complex normal distribution). However, the transformation $C\nd_{0,V_\infty}$, with non-singular $C\in\mat[C]{d}{d}$, of the real invariant normal distribution will generally follow a~degenerate complex normal distribution.
By using only block dizaonalization, and remaining in the real field, we bypass all these problems.
\end{remark}

\section{Micro-macro acceleration of linear slow-fast SDEs\label{se:mM_accel_linSDE}}

\subsection{Matching with slow mean and variance via Kullback-Leibler divergence}\label{se:match_cg}
The matching procedures introduced in Section~\ref{se:match} use the first two moments of the full distribution.
However, as we describe in Section~\ref{se:spec-drift}, in the simulation of the multi-scale systems, we are usually interested in the evolution of $d_s$ slow variables. Therefore, it is reasonable to combine the micro-macro acceleration method with coarse-graining: employing only the macroscopic states related to the slow variables in decomposition~\eqref{eq:spm_dec} during extrapolation. The matching procedures from Section~\ref{se:match} can be easily modified to encompass this case. The sole change in the formulas amounts to replacing the identity $\Pi$ on $\R^d$ with the orthogonal projection onto the slow variables, $\Pi^s\from\R^d\to\R^{d_s}$ from~\eqref{eq:spm_dec}, and using the marginal mean $\oline{\mu}^s$ and marginal covariance $\oline{\Si}^s$ in the constraints.
Therefore, the matching with slow mean only reads
\begin{equation}
\label{eq:match_smean}\tag*{$\mrm{(ME)^s}$}
\match(\oline{\mu}^s,P) = \argmin_{Q\in\prob^d}\kld(Q\|P),\ \text{constrained on}\ \Exp[Q][\Pi^s]=\oline{\mu}^s,
\end{equation}
and the matching with mean and variance becomes
\begin{equation}
\label{eq:match_smeanvar}\tag*{$\mrm{(MEV)^s}$}
\match(\oline{\mu}^s,\oline{\Si}^s,P) = \argmin_{Q\in\prob^d}\kld(Q\|P),\ \text{constrained on}\ \left\{\begin{array}{l}
\Exp[Q][\Pi^s]=\oline{\mu}^s\\
\Var[Q][\Pi^s] = \oline{\Si}^s.
\end{array}\right.
\end{equation}

For $d_s=d$, we retrieve the minimisation programs from the previous section. For $d_s<d$, this change further reduces the number of nonlinear equations for Lagrange multipliers to $d_s$ in the case of matching with marginal mean, and $d_s(d_s+3)/2$ for matching with marginal mean and variance. In this fashion, we obtain a~significant reduction in the computational effort whenever $d_s$ is much smaller that $ d$ -- an~expected feature of multi-scale systems.

We now demonstrate how the matching procedures with $d_s$-marginal moments on $\R^d$ connect back to the matchings with full moments on $\R^{d_s}$, the space of slow variables. Let $P^s$ be the marginal distribution of $P$ on the slow subspace $\R^{d_s}$. Then, for every $y\in\R^{d_s}$, there exists a unique conditional distribution $P^{f|s}({\cdot}|y)$ on $\R^{d_f}$~\cite[Thm.~10.2.2]{Dudley2002}, the distribution of the fast variables conditioned on fixing the slow ones at~$y$, and the identity $P = P^s\otimes P^{f|s}$ holds.\footnote{The identity $P = P^s\otimes P^{f|s}$ means that for every measurable rectangle $U\times V\subseteq\R^{d_s}\oplus\R^{d_f}$ and Borel function $g\from U\times V\to\R$, it holds $\int_{U\times V}g(x)\,P(\der{x})=\int_{U}\int_{V}g(y,z)\,P^{f|s}(\der{z}|y)\,P^s(\der{y})$.}  Now, we can easily see that we obtain the matching based on slow moments by (i) performing the full matching on the space of slow variables only with prior $P^s$; and (ii) multiplying the result with the conditional distribution $P^{f|s}({\cdot}|y)$ of the fast variables, given the slow ones. We detail this observation for~\ref{eq:match_smean}; the case of~\ref{eq:match_smeanvar} follows similarly.

\begin{proposition}\label{pro:match_cg}
Let $P\in\prob^d$ be a prior distribution and $\oline{\mu}^s\in\R^{d_s}$ a slow marginal mean. Then it holds
\begin{equation*}
\match(\oline{\mu}^s, P) = \match(\oline{\mu}^s, P^s)\otimes P^{f|s},
\end{equation*}
where $\match(\oline{\mu}^s, P)$ solves~\ref{eq:match_smean} and $\match(\oline{\mu}^s, P^s)$ solves~\eqref{eq:match_mean} from page~\pageref{eq:match_mean} with $d=d_s$.
\end{proposition}

\begin{proof}
Consider the log-partition function $A(\la,P)$, where $\la\in\R^{d_s}$, related to~\ref{eq:match_smean}. Employing~\cite[Thm.~10.2.1]{Dudley2002}, we compute
\begin{gather}\label{eq:logpart_marg}
\begin{aligned}
A(\la, P)
&= \ln\int_{\R^d}\exp\!\big(\la\sdot\Pi^sx\big)\,P(\der{x})
= \ln\!\Big\{\int_{\R^{d_s}}\int_{\R^{d_f}}\exp\!\big(\la\sdot y\big)\,P^{f|s}(\der{z}|y)P^s(\der{y})\Big\}\\
&= \ln\!\Big\{\int_{\R^{d_s}}\exp\!\big(\la\sdot y\big)P^s(\der{y})\Big\}%
= \ln\Exp[P^s]\big[\exp(\la\sdot\Pi^s)] = A(\la, P^s),
\end{aligned}
\end{gather}
where $A(\la,P^s)$ is the log-partition function of~\eqref{eq:match_mean} with $d=d_s$. Therefore, the vector of Lagrange multipliers $\oline{\la}\in\R^{d_s}$ of $\match(\oline{\mu}^s,P)$, the coarse-grained matching~\ref{eq:match_smean}, solves
\begin{equation*}
\grad[\la]A(\oline{\la}, P^s) = \oline{\mu}^s,
\end{equation*}
so it corresponds to $\match(\oline{\mu}^s,P^s)$, the full matching on the space $\R^{d_s}$ of slow variables.

If $Q=\match(\oline{\mu}^s,P)\in\prob^d$ and $Q^s=\match(\oline{\mu}^s,P^s)\in\prob^{d_s}$, substituting $y=\Pi^sx$ for $x$ in the right-hand side of~\eqref{eq:dist_mean} we get the following identity
\begin{equation}\label{eq:match_cg}
\rnder{Q}{P}(y,z) = \exp\big(\oline{\la}\sdot y - A(\oline{\la},P^s)\big) = \rnder{Q^s}{P^s}(y).
\end{equation}
Thus, the marginal density $(\rnder{Q}{P})^s$ equals $\rnder{Q^s}{P^s}$ and the conditional $(\rnder{Q}{P})^{f|s}({\cdot}|y)$ is constant, equal to $1$, for all $y\in\R^{d_s}$. Since $Q = (\rnder{Q}{P})^sP^s\otimes (\rnder{Q}{P})^{f|s}P^{f|s}$, this proves the formula in the statement of the proposition.
\end{proof}

Therefore, we can obtain the coarse-grained matching by performing the full matching on the space of slow variables $\R^{d_s}$ only, and reusing the conditional distributions of the prior $P^{f|s}({\cdot}|y)$.
Moreover, if the prior $P$ has density $p$ with respect to the Lebesgue measure, $P=p\der{x}$, the matching density reads $q(x) = \exp\!\big(\oline{\la}\sdot\Pi^sx-A(\oline{\la},P)\big)p(x)$, and, using~\eqref{eq:match_cg}, we have
\begin{equation}\label{eq:match_cg_dens}
q(y,z) = \exp\!\big(\oline{\la}\sdot y-A(\oline{\la},P^s)\big)p^s(y)\,p^{f|s}(z|y) = q^s(y)\,p^{f|s}(z|y),
\end{equation}
where $q^s$ is the density on $\R^{d_s}$ of $\match(\oline{\mu}^s,P^s)$.

\subsection{Algorithm formulation for linear slow-fast SDEs}\label{se:mM_accel_alg_linSDE}
In this section, we specify the micro-macro acceleration method, presented in Section~\ref{se:mM_accel}, to the simulation of linear SDE~\eqref{eq:test} with slow-fast drift matrix satisfying~\eqref{eq:test_mscale}.
The extrapolation and matching stages utilize macroscopic states of slow variables on $\R^{d_s}$, $d_s\leq d$. (When $d_s=d$, we encompass the micro-macro acceleration on full state space.)

\textit{Stage 1: Propagation of microscopic laws.}
As an \emph{inner integrator}, which propagates equation~\eqref{eq:test} forward in time, we use the \emph{Euler-Maruyama} scheme
\begin{equation}\label{eq:em_test}
\prop^{\de t}(A, B, X^{\de t}_{n,k},\de W_{k+1}) = (\Id+\de tA)X^{\de t}_{n,k} + \sqrt{B}\de W_{k+1}.
\end{equation}

\textit{Stage 2: Restriction to a~finite number of observables.}
We consider two cases of restriction that give slow macroscopic states as $\mbf{m}^s_{n,k} = \res^s(P^{\de t}_{n,k})$. We use either the slow mean
\begin{equation}\label{eq:res_smean}
\res^s(Q)=\Exp[Q][\Pi^s],
\end{equation}
and then $L=d_s$, or the slow mean and variance
\begin{equation}\label{eq:res_smeanvar}
\res^s(Q) = (\Exp[Q][\Pi^s], \Var[\!Q][\Pi^s]),
\end{equation}
with $L=d_s(d_s+3)/2$, where $\Pi^s$ is the projection on the space of slow variables.

\textit{Stage 3: Extrapolation of macroscopic states.}
In the following, we solely use forward Euler extrapolation based on the first and the last macroscopic states
\begin{equation}\label{eq:lin_extrap}
\mbf{m}^s_{n+1} = \mbf{m}^s_{n,K} + \frac{\De t-K\de t}{K\de t}(\mbf{m}^s_{n,K}-\mbf{m}^s_{n,0}).
\end{equation}

\textit{Stage 4: Matching.}
We use the matching procedure described in Section~\ref{se:match_cg}
\begin{equation*}
\match(\mbf{m}^s_{n+1},P^{\de t}_{n,K}) = \argmin_{Q}\kld(Q \| P^{\de t}_{n,K}),\ \text{constrained on}\ \res^s(Q)=\mbf{m}^s_{n+1}.
\end{equation*}

\subsection{Invariant distribution of the micro-macro acceleration method}\label{se:mM_invdistr}
In this section, we demonstrate that the micro-macro acceleration method possesses the same invariant distribution as the underlying inner integrator (the Euler-Maruyama method): a centred Gaussian $\nd_{0,V_\infty^{\de t}}$ with variance (compare with~\eqref{eq:test_equivar})
\begin{equation*}
V_\infty^{\de t} = \de t\sum_{j=0}^{\infty}(\Id+\de tA)^jB(\Id+\de t\tp{A})^j,
\end{equation*}
despite extrapolating only the slow macroscopic states of distributions in Stage~3 of the algorithm.
The sum on the right-hand side converges whenever $\spm{\Id+\de tA}\subset\dsk(0,1)$, where we denote by $\dsk(z,r)$ the disk in the complex plane with center $z$ and radius $r$. Notice that the variance depends on the micro time step but not on the macro time step $\De t$.

Having this invariant measure allows us to correctly formulate the numerical counterpart of asymptotic stability: whenever the microscopic time step $\de t$ is stable, i.e.~$\spm{\Id+\de tA}\subset\dsk(0,1)$, we seek for a condition on the macroscopic time step $\De t$ such that
\begin{equation*}
P^{\De t}_n\convto\nd_{0,V_\infty^{\de t}},\quad\text{as}\ n\goesto\infty,
\end{equation*}
where $P^{\De t}_n$ is the law generated by the micro-macro acceleration method.

To prove the invariance, let us assume that the initial distribution at time $t_n$ satisfies $P_n^{\De t}=\nd_{0,V_\infty^{\de t}}$ and trace the propagation of this distribution through the stages of the algorithm to demonstrate that $P_{n+1}^{\De t} = \nd_{0,V_\infty^{\de t}}$ as well. In a nutshell: though after restriction to slow moments the algorithm has no control over the fast marginal of the solution, since $\nd_{0,V_\infty^{\de t}}$ is invariant measure of the inner integrator, the extrapolation becomes constant and the matching, due to its projective property, reproduces to the prior which is exactly $\nd_{0,V_\infty^{\de t}}$. In more details:

\textit{Stage 1: Propagation of microscopic laws.}
Since $X^{\de t}_{n,0}\sim\nd_{0,V_\infty^{\de t}}$ and this is an invariant measure of the Euler-Maruyama scheme, the laws of $X^{\de t}_{n,k}$ do not change and it holds
\begin{equation*}
P_{n,1}^{\de t} = \ldots = P_{n,K}^{\de t} = \nd_{0,V_\infty^{\de t}}.
\end{equation*}

\textit{Stage 2: Restriction to a~finite number of slow observables.}
In consequence, the macroscopic states $\mbf{m}^s_{n,k}=\res^s(P^{\de t}_{n,k})$ will be equal for all $k=0,\dotsc,K$: either $0$ for the restriction to the slow mean~\eqref{eq:res_smean} or $(0,V_\infty^{\de t,s})$, where $V_\infty^{\de t,s}$ is the variance of the slow marginal, for the restriction to the slow mean and variance~\eqref{eq:res_smeanvar}.

\textit{Stage 3: Extrapolation of macroscopic states.}
Particularly, $\mbf{m}^s_{n,0} = \mbf{m}^s_{n,K}$ and thus the linear extrapolation~\eqref{eq:lin_extrap} yields $\mbf{m}^s_{n+1} = \mbf{m}^s_{n,K}$.

\textit{Stage 4: Matching.}
Finally, the matching $\match(\mbf{m}^s_{n+1},P_{n,K}^{\de t})$ is the identity because the constraints are already satisfied by the prior and thus the optimization procedures~\ref{eq:match_smean} and~\ref{eq:match_smeanvar} from page~\pageref{eq:match_smean} attain their minimum at $P_{n,K}^{\de t}$. Therefore,
\begin{equation*}
P^{\De t}_{n+1}=P_{n,K}^{\de t}=\nd_{0,V_\infty^{\de t}}.
\end{equation*}

In Section~\ref{se:mM_stab_gauss}, we initialise the micro-macro acceleration scheme with Gaussian random variables and establish the convergence of mean and variance of the numerical solution. Similarly as before, convergence of the first two moments yields the convergence to invariant distribution in Kullback-Leibler divergence.

\section{Stability of micro-macro acceleration with Gaussian initial conditions}\label{se:mM_stab_gauss}

We are now ready to discuss the stability of the micro-macro acceleration method when starting from Gaussian initial conditions. In Section~\ref{se:match_smean_gauss}, we first discuss the matching itself. Afterwards, we look at a micro-macro acceleration algorithm that only extrapolates the mean (Section~\ref{se:mM_gauss_smean}). Finally, we consider extrapolation of both mean and variance (Section~\ref{se:mM_gauss_smeanvar}).

\subsection{Matching a~multivariate normal distribution with marginal mean}\label{se:match_smean_gauss}
When we match a~Gaussian prior with a~new mean, the minimum Kullback-Leibler divergence matching produces a~normal distribution with the same variance and the new mean swapped for the original one~\cite{CovTho2005,Kullback1978}. In this section, we extend this fact to the matching with a~marginal mean -- when we consider only a~part of variables.

Let $\nd_{\mu,\Si}$ be the normal distribution on $\R^{d}$, with $d=d_s+d_f$, having vector mean and positive-definite covariance matrix (in block form)
\begin{equation*}
\mu =
\begin{bmatrix}
\mu^s\\
\mu^f
\end{bmatrix},\quad
\Si=
\begin{bmatrix}
 \Si^s & \mrm{C} \\[0.5em]
 \tp{\mrm{C}} & \Si^f
 \end{bmatrix},
\end{equation*}
where $\mu^s\in\R^{d_s}$, $\mu^f\in\R^{d_f}$ are the means and $\Si^s\in\R^{d_s\times d_s}$, $\Si^f\in\R^{d_f\times d_f}$ are the variance matrices of the marginal distributions, and $\mrm{C}\in\R^{d_s\times d_f}$ contains the cross-covariances between fast and slow variables. Here, we use the notation from Section~\ref{se:mM_accel_linSDE}, where $d_s$ stands for the number of slow variables (cardinality of $\Om^s$) and $d_f$ for the number of fast ones.

Consider the restriction operator $\res^s$ that computes the $s$-marginal mean in $\R^d$, corresponding to the slow variables. That is, given a distribution $P\in\prob^d$, the restriction evaluates
\begin{equation*}
\res^s(P) = \Exp[P][\proj^s],
\end{equation*}
where $\proj^s\from\R^d\to\R^{d_s}$ is the orthogonal projection onto the slow variables from Section~\ref{se:block_diag}.

\begin{proposition}\label{pro:match_smean_gauss}
Let $\match$ be the minimum Kullback-Leibler divergence matching~\ref{eq:match_smean}, from page~\pageref{eq:match_smean}, associated to $\res^s$. Then, for any given $s$-marginal mean vector $\oline{\mu}^s\in\R^{d_ s}$
\begin{equation}\label{eq:match_smean_gauss}
\match(\oline{\mu}^s,\nd_{\mu,\Si}) = \nd_{\oline{\mu},\Si},
\end{equation}
where $\oline{\mu}=\tp{[\oline{\mu}^s,\oline{\mu}^f]}$ with $\oline{\mu}^f=\mu^f + \tp{\mrm{C}}(\Si^s)^{-1}(\oline{\mu}^s-\mu^s)$.
\end{proposition}

\begin{proof}
By the law of total probability, the normal distribution $\nd_{\mu,\Si}$ on $\R^d$ decomposes into the product of the slow marginal and the fast conditional as follows (recall the notation from Proposition~\ref{pro:match_cg}):
\begin{equation*}
\nd_{\mu,\Si} = \nd_{\mu^s,\Si^s} \otimes \nd_{\mu^{f|s}(\cdot),\Si^{f|s}},
\end{equation*}
where the fast conditional mean and variance (conditioned on the slow variable $y$) read
\begin{equation}\label{eq:cond_meanvar}
\mu^{f|s}(y) = \mu^f+\tp{\mrm{C}}(\Si^s)^{-1}(y-\mu^s),\quad
\Si^{f|s} = \Si^f - \tp{\mrm{C}}(\Si^{s})^{-1}\mrm{C}.
\end{equation}

Fix a~new (slow) marginal mean $\oline{\mu}^s\in\R^{d_s}$. Proposition~\ref{pro:match_cg} gives
\begin{equation*}
\match\big(\oline{\mu}^s,\nd_{\mu,\Si}\big) = \match\big(\oline{\mu}^s,\nd_{\mu^s,\Si^s}\big)\otimes \nd_{\mu^{f|s}(\cdot),\Si^{f|s}}.
\end{equation*}
Since $\proj^s$ is the identity on $\R^{d_s}$, the matching $\match\big(\oline{\mu}^s,\nd_{\mu^s,\Si^s}\big)$ corresponds to~\eqref{eq:match_mean} from page~\pageref{eq:match_mean} with $d=d_s$ and yields the normal density with prior variance $\Si^s$ and new mean $\oline{\mu}^s$~\cite{CovTho2005,Kullback1978}. Therefore, we arrive at
\begin{equation}\label{eq:match_decomp}
\match\big(\oline{\mu}^s,\nd_{\mu,\Si}\big) = \nd_{\oline{\mu}^s,\Si^s} \otimes \nd_{\mu^{f|s}(\cdot),\Si^{f|s}}.
\end{equation}
By adding and subtracting $\tp{\mrm{C}}(\Si^s)^{-1}\oline{\mu}^s$ in the conditional mean~\eqref{eq:cond_meanvar} we get
\begin{equation}
\mu^{f|s}(y) = \mu^f + \tp{\mrm{C}}(\Si^s)^{-1}(\oline{\mu}^s-\mu^s) +\tp{\mrm{C}}(\Si^s)^{-1}(y-\oline{\mu}^s).
\end{equation}
Denoting $\oline{\mu}^f=\mu^f + \tp{\mrm{C}}(\Si^s)^{-1}(\oline{\mu}^s-\mu^s)$ and comparing with~\eqref{eq:cond_meanvar}, we can interpret $\mu^{f|s}(y)$ as the conditional mean of the vector
\begin{equation*}
\oline{\mu} =
\begin{bmatrix}
\oline{\mu}^s\\
\oline{\mu}^f
\end{bmatrix}\in\R^{d},
\end{equation*}
with the marginal variance $\Si^s$. Accordingly, the left-hand side of identity~\eqref{eq:match_decomp} is the law of total probability for the~normal density on $\R^{d}$ with mean $\oline{\mu}$ and covariance $\Si$. From this, \eqref{eq:match_smean_gauss} follows.
\end{proof}

\begin{remark}[On independent variables]
When the marginals of $\nd_{\mu,\Si}$ are uncorrelated ($\mrm{C}=0$), and thus independent, the matching results in a normal density with $\oline{\mu}^s$ substituted for $\mu^s$ in the prior mean $\mu$, that is $\oline{\mu} = \tp{(\oline{\mu}^s, \mu^f)}$. In this case, no correction to the (fast) marginal mean $\mu^f$ is needed. However, this situation is uninteresting in the multiscale framework, where the slow and fast variables are always dependent.
\end{remark}

\begin{remark}[On extension to Gaussian mixtures]\label{re:Gmix}
Proposition~\ref{pro:match_smean_gauss}, as well as Theorem~\ref{th:mM_gauss_smean} from the next section, can be extended to priors (and initial distributions) given by a \emph{Gaussian mixture}. These are distributions $P = \sum_{j=1}^{J}w_j\nd_{\mu_j,\Si_j}$, where $(w_1,\dotsc,w_J)$ is a vector of weights. The extension is straightforward but the computations become tedious. Nevertheless, since Gaussian mixtures are weakly dense in the space of all probability measures, such generalization significantly extends the scope of aforementioned results. Though we do not pursue this track in the paper, we include in Section~\ref{se:mM_nongauss_smenvar_num} a numerical illustration looking at the convergence of micro-macro acceleration method initiated with Gaussian mixture.
\end{remark}

\subsection{Stability bounds for the extrapolation of the marginal mean}\label{se:mM_gauss_smean}
\subsubsection{Analytical result}\label{se:mM_gauss_smean_ana}
Consider a~linear SDE in $\R^{d}$, $d=d_s+d_f$,
\begin{equation}\label{eq:test_bdiag}
\der{X}_t = DX_t\der{t} + \sqrt{B}\der{W}_t,
\end{equation}
where $B\in\R^{d\times d}$ and
\begin{equation}\label{eq:test_bdiag_struct}
D = \begin{bmatrix}
D^s & 0\\
0 & D^f
\end{bmatrix},\quad D^s\in\R^{d_s\times d_s},\ D^f\in\R^{d_f\times d_f}.
\end{equation}
According to the discussion in Section~\ref{se:block_diag}, test equation~\eqref{eq:test_bdiag}-\eqref{eq:test_bdiag_struct} already encompasses all possible behaviors of the general test equation~\eqref{eq:test} under Assumption~\ref{as:test_mscale}. In this fashion, block-diagonal systems constitute a convenient yet representative model for first stability analysis. Note also that, since $\spm{D^s}=\spm{D}\cap\Om^s$ and $\spm{D^f}=\spm{D}\cap\Om^f$, the block-diagonal form of $D$ conforms to its spectral decomposition into slow and fast eigenvalues.

We use the forward Euler extrapolation~\eqref{eq:lin_extrap} of the marginal mean in~$\R^{d_s}$ obtained from restriction~\eqref{eq:res_smean}, based on one Euler-Maruyama step of size~$\de t$ (thus $K=1$), to propagate it over a time interval of size~$\De t$. We initialize with the normal distribution $\nd_{\mu_0,\Si_0}$. Since both the microscopic scheme and extrapolation/matching preserve the Gaussianity of the initial distribution, we focus on the evolution of the mean and variance only. In the following theorem, $\spr{\bullet}$ denotes the \emph{spectral radius of a~matrix} -- the largest absolute value of its eigenvalues.

\begin{theorem}\label{th:mM_gauss_smean}
When applying the micro-macro acceleration method to the linear SDE~\eqref{eq:test_bdiag} with block-diagonal structure~\eqref{eq:test_bdiag_struct}, the mean $\mu_n$ and covariance $\Si_n$ of the resulting Gaussian law at $n$th step satisfy
\begin{equation*}
\lim_{n\to+\infty}\mu_n=0,\quad \lim_{n\to+\infty}\Si_n=V_\infty^{\de t},
\end{equation*}
whenever
\begin{equation}\label{ass:spectral_gap_discrete}
\spr{\Id^s+\De tD^s}<1 \qquad \text{and} \qquad \spr{\Id^f+\de tD^f}<1.	
\end{equation}
\end{theorem}

\begin{remark}\label{re:mM_gauss_smean}
\begin{enumerate}[(a)]
\item
Since the constraints on the time steps correspond to different modes, Theorem~\ref{th:mM_gauss_smean} demonstrates, in a~simple test case, that the micro-macro acceleration method breaks the stability barrier for the stiff system~\eqref{eq:test_bdiag}. More precisely, if we put $\rh^s=\spr{D^s}$ and $\rh^f=\spr{D^f}$, the bounds on the spectral radii become $\De t\lesssim 1/\rh^s$ and $\de t\lesssim 1/\rh^f$. By Assumption~\ref{as:test_mscale} $\rh^s\ll\rh^f$, and thus the stiff bound acts on the micro step $\de t$ only, whereas the extrapolation step $\De t$ experiences a~much milder restriction.

\item
Moreover, the stability bound for the extrapolation scheme coincides with the stability threshold of the macroscopic closure for the slow mean. Indeed, the closed ODE for the slow mean of~\eqref{eq:test_bdiag} reads
\begin{equation*}
\dot{m}^s(t) = D^sm^s(t),
\end{equation*}
and the stability bound for the Euler scheme applied to this equation is $\De t\lesssim 1/\rh^s$.

\item
Note also that the bias in variance, always present when discretizing in time, depends only on $\de t$. Therefore, by extrapolating the (slow) marginal mean over $\De t$, we do not introduce additional bias on top of the one due to the Euler-Maruyama scheme. We observe this property numerically in Section~\ref{se:mM_gauss_smean_num} and compare with the extrapolation of (slow) marginal mean and variance in Section~\ref{se:mM_gauss_smeanvar}.
\end{enumerate}
\end{remark}

\begin{proof}
Let us fix $n\geq 0$ and let $\nd_{\mu_n,\Si_n}$ be the law at time $t_n$. We apply only one time step of the inner Euler-Maruyama scheme~\eqref{eq:em_test}, thus the mean and variance of the~prior are
\begin{equation*}
\mu_{n,1} = (\Id+\de tD)\mu_{n},\quad
\Si_{n,1} = (\Id+\de tD)\Si_n\tp{(\Id+\de tD)} + \de tB.
\end{equation*}
Moreover, the extrapolation formula~\eqref{eq:lin_extrap} becomes
\begin{equation}\label{eq:extr_smean}
\mu_{n+1}^s
= \mu^s_{n,1} +\frac{\De t - \de t}{\de t}\big(\mu^s_{n,1}-\mu^s_n\big)
= (\Id^s + \De tD^s)\mu^s_n,
\end{equation}
since we restrict with the (slow) marginal mean only. %
According to Proposition~\ref{pro:match_smean_gauss}, matching the extrapolated (slow) marginal mean $\mu^s_{n+1}$ with the prior $\nd_{\mu_{n,1},\Si_{n,1}}$ produces the normal distribution $\nd_{\mu_{n+1},\Si_{n+1}}$ with covariance matrix $\Si_{n+1}=\Si_{n,1}$ and vector mean
\begin{gather}\label{eq:extr_mean}
\begin{aligned}
\mu_{n+1}
&= \begin{bmatrix}
(\Id^s + \De tD^s)\mu^s_n\\[0.5em]
(\Id^f+\de tD^f)\mu^f_n + (\De t-\de t)\tp{C_{n,1}}(\Si^s_{n,1})^{-1}D^s\mu^s_n
\end{bmatrix}\\[1em]
&= \Bigg(\Id + \begin{bmatrix}
\De tD^s & 0\\[0.5em]
(\De t-\de t)\tp{C_{n,1}}(\Si^s_{n,1})^{-1}D^s & \de tD^f
\end{bmatrix}\Bigg)\mu_n \doteq (\Id + D_n)\mu_n.
\end{aligned}
\end{gather}

As a~result, as long as the microscopic method remains stable, the covariance matrix $\Si_n$ of the method stays bounded and converges to $V_\infty^{\de t}$ when $n$ grows to infinity. For the Euler-Maruyama scheme, this is guaranteed by the inequality $\spr{\Id+\de tD}<1$. Thus, the bias in the asymptotic variance of the micro-macro acceleration method depends only on the microscopic time step $\de t$.%

For the mean, note first that $\spm{D_n} = \spm{\De tD^s}\cup\spm{\de tD^f}$, because $D_n$ is block lower-triangular.
Since the spectrum of the matrix $\Id+D_n$ does not depend on the current step,  the recurrence~\eqref{eq:extr_mean} is stable (that is $\lim_{n\to\infty}\mu_n=0$) when
\begin{equation*}
\max\big\{\rh(\Id^s+\De tD^s),\ \rh(\Id^f+\de tD^f)\big\}<1.
\end{equation*}
This shows that the asymptotic stability of the mean prevails, regardless of the spectral gap, for $\De t$ and $\de t$ satisfying assumption~\eqref{ass:spectral_gap_discrete}.
\end{proof}

\subsubsection{Numerical illustrations}\label{se:mM_gauss_smean_num}
We complement and extend Section~\ref{se:mM_gauss_smean_ana} with numerical experiments in which the extrapolation time step $\De t$ crosses its stability bound. First, we show that the presence of matching failures, see Remark~\ref{re:match_fail}, indicates, even before blow-up, that the extrapolation time step lies beyond the stability threshold. Second, using matching failures as stability criterion, we explore the $(\de t,\De t)$-parameter space against the stability of the micro-macro acceleration method. We do this for two linear systems having the same drift spectrum: one with diagonal drift matrix, which falls within the scope of Theorem~\ref{th:mM_gauss_smean}, and the other one with slow-fast structure, which displays more intricate behavior. Third, we discuss an adaptive extrapolation time-stepping strategy based on matching failures that allows to perform the micro-macro accelerations when the exact stability threshold remains unknown.

As our first model, we consider the linear system with diagonal drift matrix
\begin{equation}\label{eq:diag_num}
D = \begin{bmatrix}
-1 & 0\\
0 & -10
\end{bmatrix},\quad
\otilde{B} = \frac{1}{9}\begin{bmatrix}
9 & \sqrt{10}\\
\sqrt{10} & 20/9
\end{bmatrix},
\end{equation}
which arises by drift-dizaonalization, as described in Section~\ref{se:block_diag}, of our second model, a linear slow-fast system~\eqref{eq:test_sf} with matrices
\begin{equation}\label{eq:sf_num}
A = \begin{bmatrix}
-1 & 1\\
0 & -10
\end{bmatrix},\quad
B = \begin{bmatrix}
1 & 0\\
0 & 10
\end{bmatrix}.
\end{equation}
The coefficients in~\eqref{eq:sf_num} correspond to the slow-fast  equation~\eqref{eq:test_sf} with parameters $a_{11}=-1$, $a_{12}=1$, $a_{22}=-1$ and $\vep=1/10$, see also Example~\eqref{ex:test_sf}. The systems~\eqref{eq:sf_num} and~\eqref{eq:diag_num} only differ by the coordinate transform that ensures the block-dizaonalization. In both systems, the spectrum of drift equals $\{\ka^s=-1, \ka^f=-10\}$ and thus the stability bound on the micro time step $\de t$ is $2/|\ka^f|=0.2$.

In the numerical experiments, we perform the micro-macro acceleration with extrapolation of the slow mean, as described in Section~\ref{se:mM_accel_alg_linSDE}. We always use one micro time step $\de t$ for the inner Euler-Maruyama integrator, thus setting $K=1$. For the Monte Carlo simulation, described in~Appendix~\ref{se:MC_sim}, we fix the number of replicas to $J=5\cdot 10^4$ and initialise with i.i.d.~samples from the  normal distribution.

\begin{remark}[On matching failure]\label{re:match_fail}
In the numerical implementation of \textit{Stage 4}, we cannot guarantee that the Newton-Raphson iteration to find Lagrange multipliers (by solving~\eqref{eq:lagr_MC} in the Monte Carlo discretization) reaches the desired tolerance, especially when, for efficiency, we limit the maximum number of iterations in the matching stage. Whenever the Newton-Raphson iteration does not converge within the prescribed number of iterations, we call this situation \emph{matching failure}. The natural strategy to deal with matching failures during the Monte Carlo simulation of the micro-macro acceleration method is to use adaptive extrapolation time-stepping. Later in this Section, we discuss the application of adaptive extrapolation time-stepping to the study of stability of the micro-macro acceleration method.
\end{remark}

\textbf{Detecting instability with matching failures.}
In a first experiment, we look at the slow marginal of drift-diagonal system~\eqref{eq:diag_num} at final time $T=210$ obtained for a range of macro time steps $\De t$ that goes beyond the stability bound $\De t^* = 2$, which results from Theorem~\ref{th:mM_gauss_smean}. To asses the accuracy, or lack thereof, we compare these slow marginals to the slow marginal of the invariant distribution of the micro-macro acceleration method for system~\eqref{eq:diag_num}. The full invariant distribution is Gaussian with mean $0$ and variance $V_\infty^{\de t}$ (which we approximate computing a partial sum of~\eqref{eq:test_equivar}), corresponding to the inner Euler-Maruyama scheme with time step $\de t$, see Section~\ref{se:mM_invdistr}. Thus, the slow marginal of the invariant distribution is normally distributed with zero mean and variance $V_\infty^{s,\de t}$ given by the top left entry of $V_\infty^{\de t}$.

Since we use a fixed extrapolation step $\De t - \de t$, with the micro time step $\de t$ set to $0.09$, we cannot guarantee, especially for $\De t$ beyond stability bound, that the Newton-Raphson procedure for the matching converges with a given tolerance within a fixed number of iterations, see Remark~\ref{re:match_fail}. In this experiment, we set the maximum number of iterations to $50$, to give the Newton-Raphson procedure enough time to converge. When we still do not reach the given tolerance, we set the weights using the last available Lagrange multipliers. Hence, throughout this experiment we proceed with the simulation even if the weights obtained after matching are incorrect. We call this situation \emph{matching failure} and record the number of times it happened during the simulation. For macroscopic time steps below $\De t = 2$ we never encountered matching failures.

In Figure~\ref{fig:mMnoad_mean_nostab}, we plot histogram densities resulting from the micro-macro acceleration method and compare them with the contour of the invariant distribution. Below $\De t=2$ the histogram density aligns well with this contour; we also do not record any matching failures. After crossing the threshold $\De t=2$, we encounter matching failures and the resulting histogram density fails to correspond to the contour any more.

\begin{figure}
\centering
\includegraphics[width=0.7\textwidth]{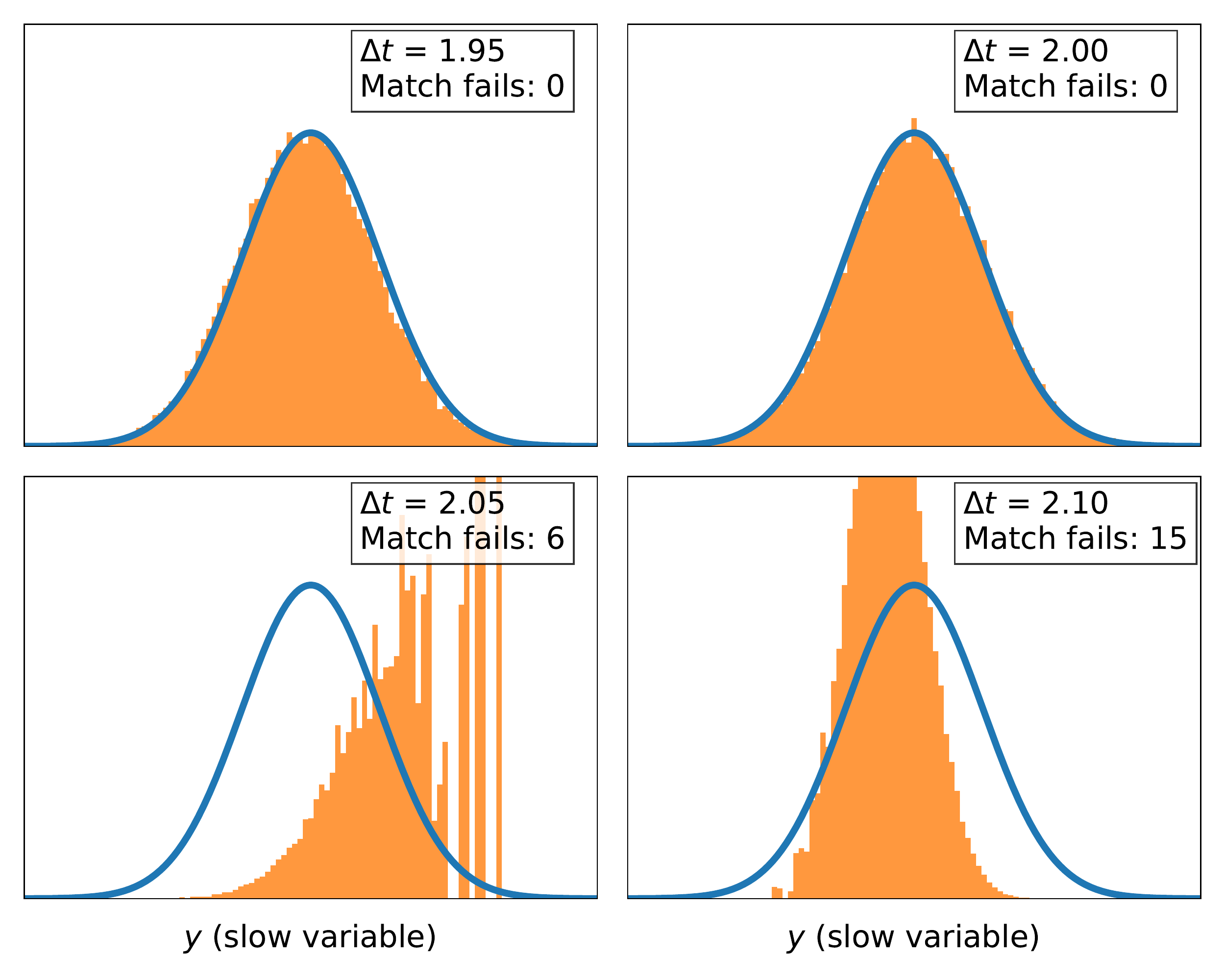}
\caption{The slow distribution (orange histograms) at time $T=210$ of the micro-macro acceleration method for~\eqref{eq:diag_num} with different extrapolation steps $\De t$ compared with the stationary Gaussian density $\nd_{0,V_\infty^{s,\de t}}$ (blue curve) of the inner Euler-Maruyama scheme based on micro step $\de t=0.09$. Below the stability threshold (equal to $2$) the slow distribution agrees well with stationary density related solely to the Euler-Maruyama scheme. The accuracy of simulation breaks when $\De t$ crosses $2$, which is also manifested by the appearance of matching failures.}
\label{fig:mMnoad_mean_nostab}
\end{figure}

\textbf{Comparing stability bounds for drift-diagonal and slow-fast systems.}
In a second experiment, using matching failures as benchmark, we compare the stability of the two systems~\eqref{eq:sf_num} and~\eqref{eq:diag_num} as a function of combined micro $\de t$ and macro $\De t$ time steps. The effect of the coordinate transform that gives~\eqref{eq:diag_num} its block-diagonal structure is that the extrapolated slow variable in both systems is different.

We perform micro-macro simulations as in the previous test for a number of grid points $(\de t, \De t)$ in the rectangle $[0, 0.2]\times[0.2, 2.3]$. For the micro time step $\de t$, we choose a range of values below the stability threshold $\de t=0.2$, so we always keep the inner integrator stable. For the range of macro time steps $\De t$, the smallest chosen value corresponds to $\De t = \de t$, i.e., no extrapolation, and the largest chosen value goes beyond the stability threshold $\De t^*=2$ from Theorem~\ref{th:mM_gauss_smean}. We classify every grid point either as unstable (orange stars), as soon as at least one matching failure occurred during the simulation, or stable (blue dots), when no failures took place up to the final time $T=210$. We present the results in Figure~\ref{fig:mM_stab_dt_vs_Dt}.

\begin{figure}
\centering
\includegraphics[width=0.35\textwidth]{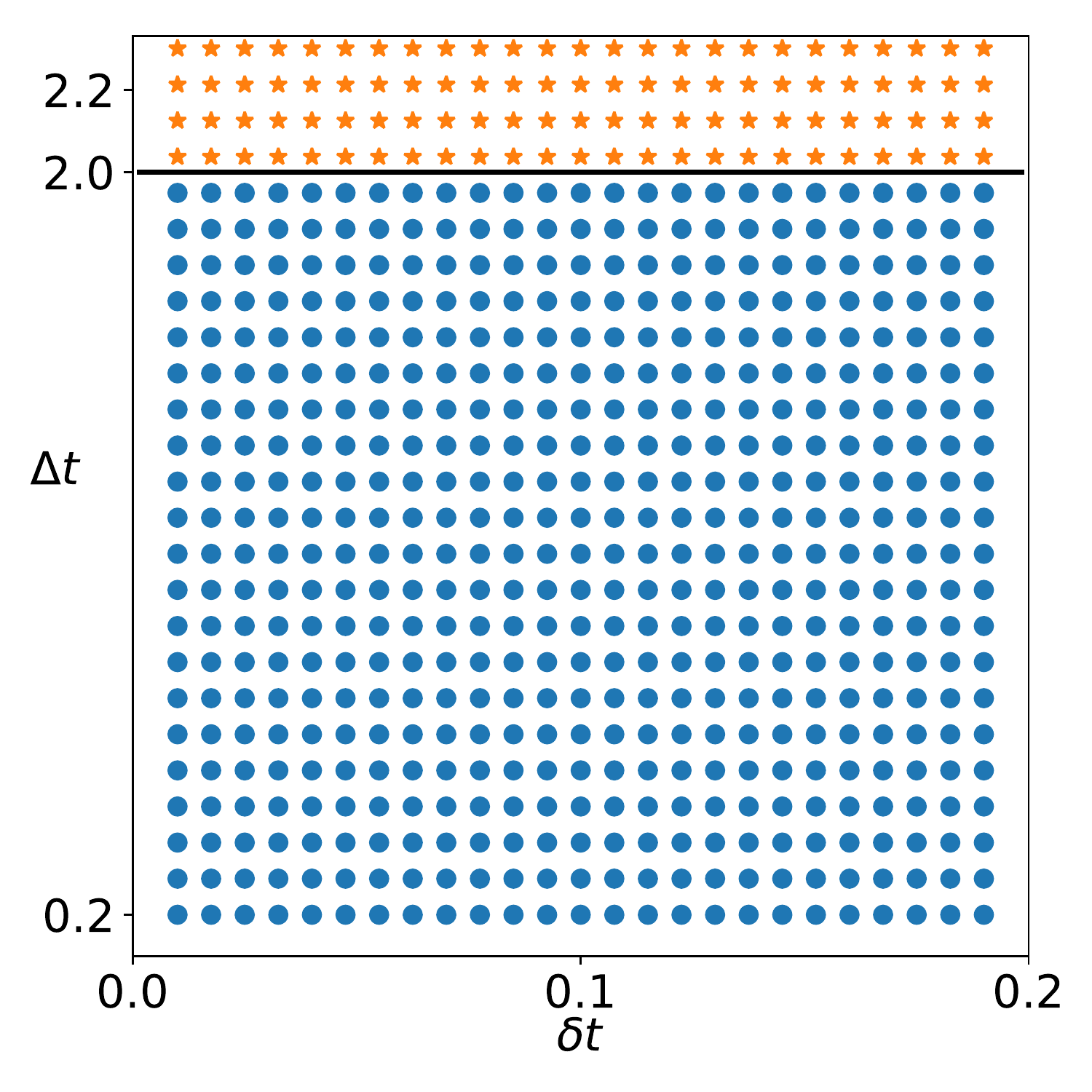}
\includegraphics[width=0.35\textwidth]{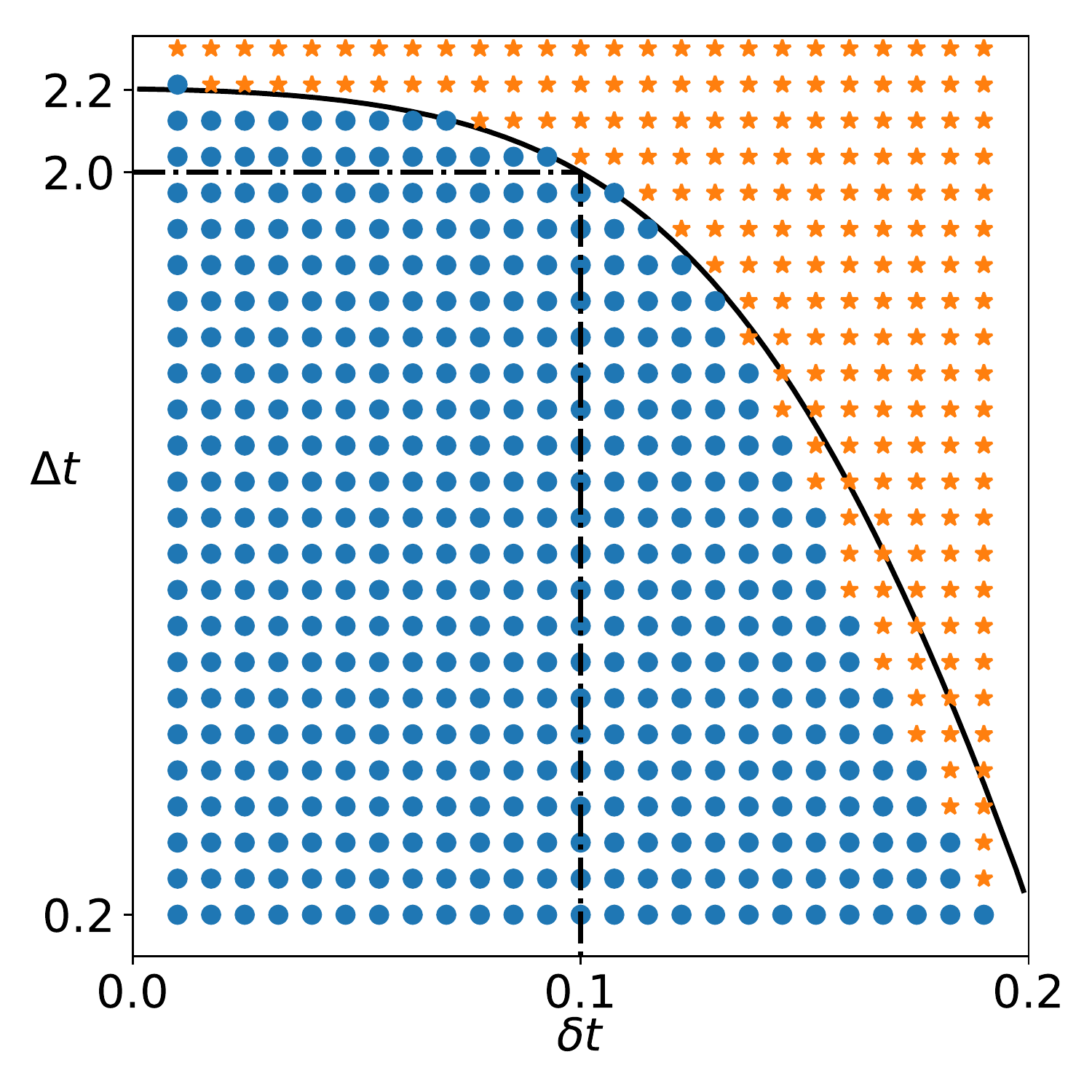}
\caption{The $(\de t,\De t)$-parameter space classified according to the stability (blue dots) or instability (orange stars) of the micro-macro acceleration method for two linear systems with the same drift spectrum. For drift-diagonal system~\eqref{eq:diag_num} (left figure), the result agrees with the conclusion of Theorem~\ref{th:mM_gauss_smean}: the macroscopic stability threshold (solid black line) stays independent of the micro time step $\de t$ and equals $2$. In the case of slow-fast system~\eqref{eq:sf_num} (right figure), the stability of extrapolation decreases as $\de t$ approaches the microscopic stability threshold $0.2$. Around $\de t=0.1$, where the macroscopic stability threshold plunges below $2$, we can approximate the borderline with the solutions of $\spr{I+A(\de t,\De t)}=1$ (solid black line), see text.}
\label{fig:mM_stab_dt_vs_Dt}
\end{figure}

First, we verify once more that detecting matching failures provides a good criterion to test the stability of micro-macro simulation. The classification of grid points in the left figure, which indicates stability of drift-diagonal system~\eqref{eq:diag_num}, aligns with the conclusion of Theorem~\ref{th:mM_gauss_smean}. The splitting of the slow and fast variables in the drift matrix corresponds to the splitting of stability thresholds for the micro and macro time steps. Second, the micro and macro stability bounds cease to be independent for the slow-fast system~\eqref{eq:sf_num} without block-diagonal structure. We can see in the right figure that simulations become unstable for smaller extrapolation steps when the micro time step $\de t$ tends to its stability threshold.

The dependence of the macroscopic stability threshold on $\de t$ results from the influence of fast modes on the extrapolated slow mean. To see this, note that when using matrix $A$ from~\eqref{eq:sf_num} instead of matrix $D$ from~\eqref{eq:diag_num} the matrix $D_n$ in~\eqref{eq:extr_mean} becomes
\begin{equation*}
A_n(\de t,\De t) = \begin{bmatrix}
-\De t & \De t -\de t\\[0.5em]
-(\De t-\de t)\tp{C_{n,1}}(\Si^s_{n,1})^{-1} & -10\de t
\end{bmatrix}.
\end{equation*}
Therefore, the matrix ceases to be lower-triangular, and its spectrum depends on the values of the covariance $C_n$ and slow-variance $\Si^s_n$ at each step of the micro-macro acceleration procedure (through the left bottom entry of $A_n$).

It is difficult to analyse the stability of $I+A_n$, because we cannot evaluate the joint spectral radius for this family of matrices~\cite{Jungers2009}. However, we can gain some insight by assuming that the variance matrix $\Si_n$ is close to the equilibrium value $V^{\de t}_\infty$. Defining $R(\de t) = \tp{(C^{\de t}_{\infty})}(V^{s,\de t}_{\infty})^{-1}$, the regression coefficient of $V^{\de t}_{\infty}$, we can focus on the matrix
\begin{equation*}
A(\de t, \De t) = \begin{bmatrix}
-\De t & \De t -\de t\\[0.5em]
-(\De t-\de t)R(\de t) & -10\de t
\end{bmatrix},
\end{equation*}
which is independent on the current step. Plotting the values of $(\de t,\De t)$ where the spectral radius of $I+A(\de t,\De t)$ equals $1$ we obtain the solid black lines in Figure~\ref{fig:mM_stab_dt_vs_Dt}. We can see that on the right figure this line agrees well with the results of simulation for small and moderate values of $\de t$; for micro time steps close to the stability threshold of the Euler-Maruyama scheme, the line deviates from the results of simulation. We can attribute this inconsistency to the fact that for large $\de t$ the values of $\Si_n$ oscillate more around the equilibrium, invalidating our assumption. Nevertheless, the line remains accurate around the values of $\de t=0.1$ where it plunges below $2$ -- the macroscopic stability threshold of the slow part only.

We can characterise the point $\de t = 0.1$ on the right figure as the value of micro time step $\de t$ for which $C^{\de t}_\infty=0$. To see this, note that when $R(\de t)=0$ the matrix $A(\de t, \De t)$ becomes upper-triangular and its spectral radius stays below $1$ for all $\De t$ smaller than $2$. The regression coefficient vanishes only when $C^{\de t}_\infty=0$. To conclude, the analysis of system~\eqref{eq:sf_num} suggests that the macro stability threshold for the micro-macro acceleration method stays independent on the time-scale separation present in the slow-fast system when we keep the micro time step $\de t$ below the value for which $C^{\de t}_\infty=0$.

\textbf{Convergence of the distribution of fast variables}
In a third experiment, we look at both slow and fast marginals of the slow-fast linear system~\eqref{eq:sf_num} obtained by the micro-macro acceleration method at the final time $T=210$ for a range of macro time steps $\De t$. We compare these marginals to the corresponding marginals of the invariant distribution of the method: Gaussian with mean $0$ and variance $V^{\de t}_{\infty}$.  The details of this experiment are similar to those in detecting instabilities with matching failures and Figure~\ref{fig:mMnoad_mean_nostab}.

In the two upper plots of Figure~\ref{fig:mMnoad_smean_faststab_linSF10}, we see that the computation was stable up to the final time -- no matching failures were recorded -- and, as previously, the slow distribution (orange histogram) agrees with the slow marginal of invariant measure (solid blue line). Additionally, the distribution of fast variables also complies with its invariant counterpart. The accuracy of the fast variables over the whole simulation time domain is not a guaranteed feature of the micro-macro acceleration method from Section~\ref{se:mM_accel_alg_linSDE} -- whose focus lies on the slow macroscopic observables. However, for a large final time horizon, Figure~\ref{fig:mMnoad_smean_faststab_linSF10} shows that the distribution of fast variables will also converge towards the correct invariant.

In the two bottom plots of Figure~\ref{fig:mMnoad_smean_faststab_linSF10}, we can see that instabilities develop, corroborated by the appearance of matching failures, and the distributions of both slow and fast variables diverge from the corresponding invariant marginals. The stability threshold is located between $\De t = 1.9$ and $\De t = 1.95$. This agrees with the findings for the linear slow-fast system~\eqref{eq:sf_num} presented in the right plot of Figure~\ref{fig:mM_stab_dt_vs_Dt}, which shows decrease in stability as the micro time step increases (here we take $\de t = 0.11$).

\begin{figure}
\centering
\includegraphics[width=\textwidth]{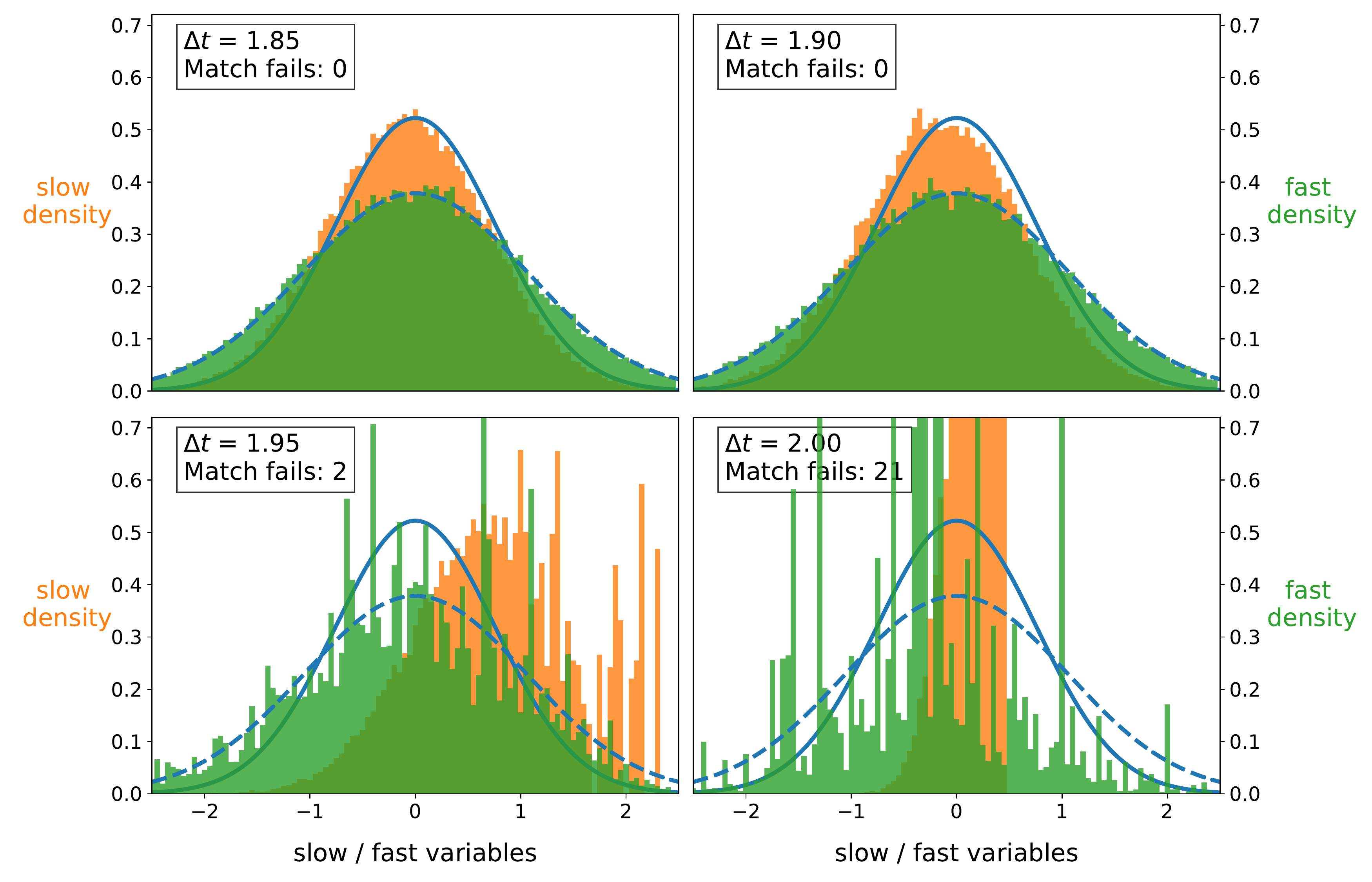}
\caption{The slow (orange) and fast (green) distribution at time $T=210$ of the micro-macro acceleration method for slow-fast system (5.11) with different extrapolation steps $\De t$ compared with the slow (solid blue curve) and fast (dashed blue curve) stationary Gaussian density $N_{0, V_{\de t}}$ of the underlying Euler-Maruyama micro scheme with time step $\de t = 0.11$. In the two upper plots, the method is stable and we can see that it reproduces correctly the fast invariant density as well. In the bottom plots, the stability of the method brakes and it fails to converge to the invariant distribution for both slow and fast variables. Compared with Figure~\ref{fig:mMnoad_mean_nostab}, we can notice that, due to larger $\de t$, the stability threshold decreased (and is somewhere between $1.9$ and $1.95$). This is compatible with the observed stability loss for the slow-fast systems as the micro time step increases, as illustrated in the right plot of Figure~\ref{fig:mM_stab_dt_vs_Dt}.}
\label{fig:mMnoad_smean_faststab_linSF10}
\end{figure}

\textbf{Adding adaptive time-stepping.}
In the final experiment with the slow mean extrapolation, we investigate the influence of crossing the stability bound on an adaptive strategy for the extrapolation time step, see also~\cite{DebSamZie2017}. The adaptive strategy reduces the extrapolation time step by a half when the Newton-Raphson procedure for the matching does not reach the desired tolerance within $10$ iterations. Whenever it converges, we enlarge the proposed extrapolation length, by a factor $1.2$, for the next step of micro-macro algorithm. The actual $\De t$ is always kept within the interval $[\de t, \De t_{\mrm{max}}]$. For the simulation we use the same parameters as in the first experiment and we apply the adaptive micro-macro acceleration method to slow-fast system~\eqref{eq:sf_num}.

\begin{figure}
\centering
\includegraphics[width=0.8\textwidth]{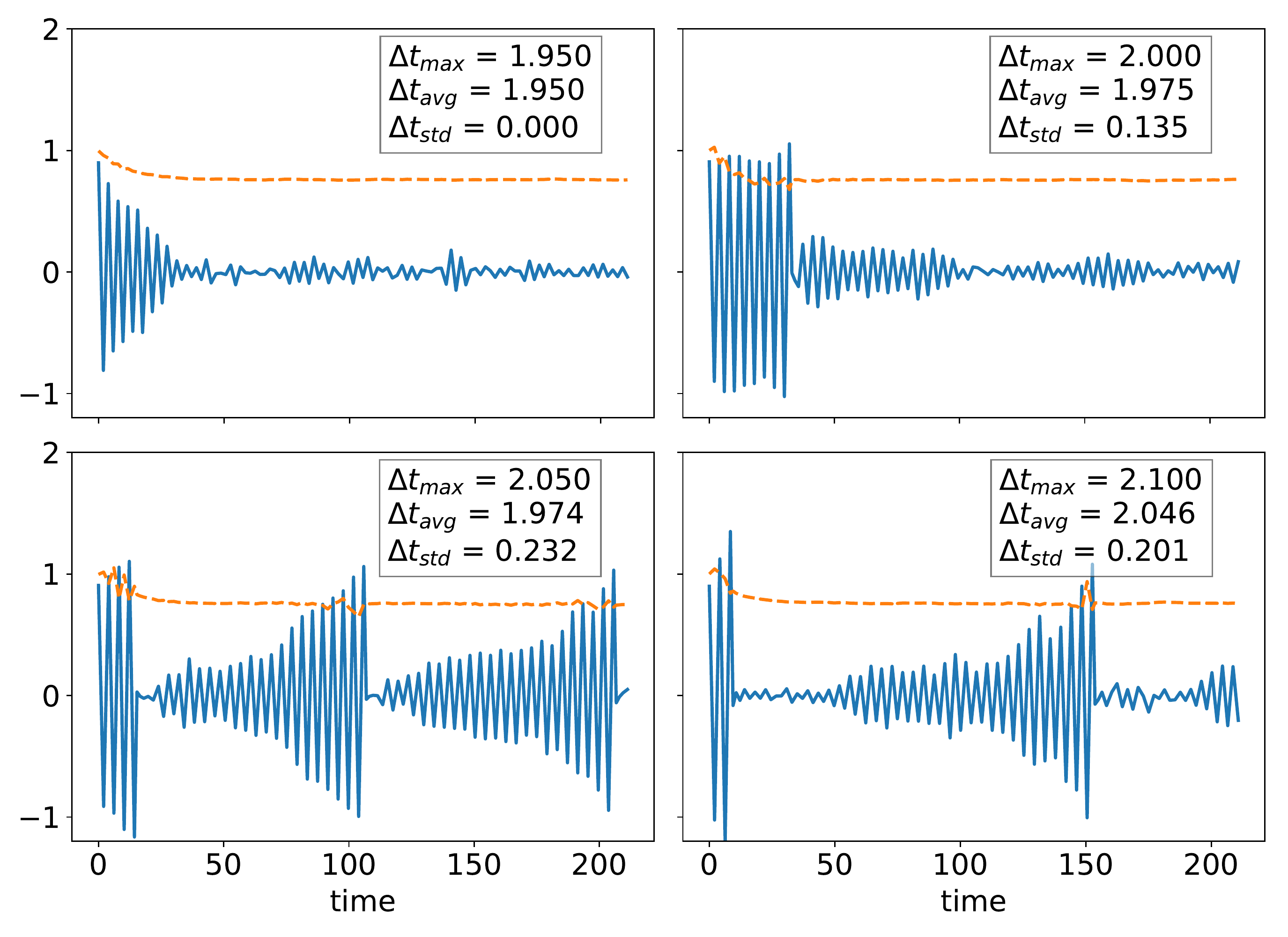}
\caption{The evolution of slow mean (solid blue) and slow standard deviation (dashed orange) of the micro-macro acceleration method with adaptive extrapolation strategy of the slow mean. The adaptive time-stepping activates when the maximal value of extrapolation step $\De t_{\mrm{max}}$ is set above the stability threshold (equal to $2$). We can see adaptivity by measuring the average extrapolation time step $\De t_{\mrm{avg}}$ over $[0,T]$ which becomes smaller than the maximal value $\De t_{\mrm{max}}$. Even though $\De t_{\mrm{avg}}$ can be larger than the stability threshold, the actual time steps drop occasionally below $2$, as their standard deviation $\De t_{\mrm{std}}$ over $[0,T]$ indicates.}
\label{fig:mMadap_smean_nostab}
\end{figure}

In Figure~\ref{fig:mMadap_smean_nostab}, we plot the resulting evolution of the slow mean and the slow standard deviation together with basic statistics of the time steps produced during the simulation over $[0,T]$. We see that the adaptive strategy reduces $\De t$ when $\De t_{\mrm{max}}$ crosses $2$, but not when it stays below $2$: activation of adaptive time-stepping indicates the crossing of the stability threshold. Moreover, the further $\De t_{\mrm{max}}$ exceeds the stability threshold, the more often the adaptive strategy is invoked. Therefore, we can safely assume that this adaptive time-stepping procedure works in situations where the stability threshold remains unknown (or changes as a function of time). We can also see that on average $\De t$ stays above the stability bound sometimes giving a good approximation to the invariant density of the EM scheme. The values of the standard deviation indicate that $\De t$ drops occasionally below~$2$.

\subsection{Stability bounds for the extrapolation of marginal mean and variance}\label{se:mM_gauss_smeanvar}
In this Section, we analyze the method in the case when besides the slow mean we also extrapolate the slow variance. In general, the addition of the second moment makes the simulation more accurate during the transient regime, see~\cite{Vandecasteele2019}. Here, we concentrate only on finding the stability bounds, and we demonstrate the they are approximately equal to the stability bound of the closure ODE for slow mean and variance only.

\textbf{Kronecker's symbols.} To study stability of the micro-macro acceleration method when extrapolating both mean and variance, we derive recurrence relations for the variance that are based on Kronecker products and sums \cite{HorJoh1991}. For two matrices $M,M'\in\mat{d}{d}$, we define their \emph{Kronecker product}, denoted $M\kdot M'$, as a linear operator on $\mat{d}{d}$ such that
\begin{equation*}
(M\kdot M')\ldot A = M'A\tp{M},\quad A\in\mat{d}{d},
\end{equation*}
where the lower dot `$\ldot$' indicates the application of an operator on a matrix, to visually distinguish it from the matrix multiplication. In what follows, we also employ the \emph{Kronecker sum} defined as $M\ksum M' = M\kdot I + I\kdot M'$. The spectrum of the Kronecker product and sum operators is the spectrum of the underlying matrices in $\mat{d^2}{d^2}$. Moreover, if $j\geq 0$ is an integer, $(M\kdot M')^j$ stands for the $j$-fold composition of $M\kdot M'$, and similarly for the Kronecker sum or any combination of both.

\subsubsection{Asymptotics of (slow) marginal mean and variance}
To study the stability of the micro-macro acceleration method specified in Section~\ref{se:mM_accel_alg_linSDE} when extrapolating mean and variance, thus we take here the restriction given by~\eqref{eq:res_smeanvar},
we again consider test equation~\eqref{eq:test_bdiag} with block diagonal drift matrix $D=\diag(D^s,D^f)$.
To conduct the same analysis as in Section~\ref{se:mM_gauss_smean_ana}, we should derive the recursion relations for the fast mean, fast variance, and covariance matrix as well. These relations incorporate the influence of matching on fast moments and are based on an appropriate modification of Proposition~\ref{pro:match_smean_gauss}.\footnote{Stating that the matching of a prior $\nd_{\mu,\Si}$ with a slow mean $\oline{\mu}^s$ and a slow variance $\oline{\Si}^s$ results in the normal distribution with the fast mean $\oline{\mu}^f=\mu^f+\tp{C}(\Si^s)^{-1}(\oline{\mu}^s-\mu^s)$ and fast variance $\oline{\Si}^f=\Si^f - \tp{C}(\Si^s)^{-1}(C-\oline{C})$ where $\oline{C}=\tp{C}(\Si^s)^{-1}\oline{\Si}^s$.} We do not present this derivation due to the complexity of the resulting formulas. Instead, we find it more insightful to focus on the slow moments only and supplement the analysis with the numerical illustration.

From~\eqref{eq:extr_smean}, we already know that the forward Euler extrapolation of the mean results in
\begin{gather*}
\begin{aligned}
\mu_{n+1}^s
&= (\Id^s + \De tD^s)\mu^s_n.
\end{aligned}
\end{gather*}
Therefore, when $\spm{\De t D^s}\subset\dsk(-1,1)$, $\mu^s_n$ goes to zero as $n$ increases, which corresponds to the behaviour of the slow mean in Section~\ref{se:mM_gauss_smean}.

After applying one step of the EM scheme, the (slow) marginal variance reads
\begin{gather*}
\begin{aligned}
\Si_{n,1} & = (\Id^s + \de t D^s)\Si^s_{n}\tp{(\Id^s + \de t D^s)} + \de tB^s.
\end{aligned}
\end{gather*}
Forward Euler extrapolation gives
\begin{gather*}
\begin{aligned}
\Si^s_{n+1}
&= \Si^s_{n} + \frac{\De t}{\de t}\big(\Si^s_{n,1}-\Si^s_n\big)= \Si^s_n + \De t\big(D^s\Si^s_n + \Si^s_n\tp{(D^s)} + \de t D^s\Si^s_n\tp{(D^s)}\big) + \De tB^s\\[0.5em]
&= \big[\Id^{s\kdot s} + \De t\big(D^s\ksum D^s + \de t D^s\kdot D^s\big)\big]\!\ldot\Si^s_n + \De tB^s.
\end{aligned}
\end{gather*}
The preceding recursion shows that the asymptotic behaviour (as $n$ increases) of $\Si^s_n$ depends on the spectrum of $L^{s\kdot s}_{\de t} \doteq D^s\ksum D^s + \de t D^s\kdot D^s$. Before investigating this spectrum, we note that, whenever $\spm{\De tL^{s\kdot s}_{\de t}}\subset\dsk(-1,1)$, the sequence of variances converges to
\begin{align*}
\De t\sum_{j=0}^{\infty}\big(\Id^{s\kdot s} +\De t L^{s\kdot s}_{\de t}\big)^{j}\!\ldot B^s = \De t\Big[\sum_{j=0}^{\infty}\big(\Id^{s\kdot s} +\De t L^{s\kdot s}_{\de t}\big)^{j}\Big]\!\ldot B^s =-\big(L^{s\kdot s}_{\de t}\big)^{-1}\!\ldot B^s.
\end{align*}
Therefore, the asymptotic (slow) marginal variance of the micro-macro acceleration method is independent of the value of the macroscopic step $\De t$ and depends only on the drift matrix $D^s$, the diffusion $B$ and the inner Euler-Maruyama time step $\de t$.

Now, we investigate the spectrum of $L^{s\kdot s}_{\de t}$. Because the micro time step $\de t$ must damp all the fast modes, Assumption~\ref{as:test_mscale} ensures that $\de t\ll\spr{D^s}$. Hence, we look at $L^{s\kdot s}_{\de t}$ as a perturbation of $D^s\ksum D^s$ by the matrix $\de tD^s\kdot D^s$. The formulas for the spectrum of Kronecker's product and sum are given by~\cite[pp.~268, 245]{HorJoh1991}
\begin{equation*}
\spm{D^s\ksum D^s} = \{\ka' + \ka'':\ \ka',\ka''\in\spm{D^s}\},\quad
\spm{D^s\kdot D^s} = \{\ka'\ka'':\ \ka',\ka''\in\spm{D^s}\}.
\end{equation*}

Disregarding the perturbation by $D^s\kdot D^s$ for a moment, we see that the asymptotic stability of the (slow) marginal variance is related to requiring $\spm{\De tD^s\ksum D^s}\subset\dsk(-1,1)$, which, by the form of the spectrum of Kronecker's sum, reads $\spm{\De tD^s}\subset\dsk(-1/2,1/2)$. Note that $\dsk(-1/2,1/2)$ equals the Euler stability region of the linear system of ODEs
\begin{equation}\label{eq:ODEclos}
\dot{m}^s(t) = D^sm^s(t),\quad
\dot{V}^s(t) = D^sV^s(t) + V^s(t)\tp{(D^s)} + B^s,
\end{equation}
for the evolution of slow mean and variance. Thus, the stability bound on the leading part~of $L^{s\kdot s}_{\de t}$ coincides with the bound of the Euler method for the moment closure with time step~$\De t$.

To estimate the influence of the perturbation, we employ the Bauer--Fike theorem~\cite[Thm.~6.3.2]{HorJoh2013} which, for any $\ka\in\spm{L^{s\kdot s}_{\de t}}$, gives an estimate
\begin{equation*}
\min_{\ka'\in\spm{D^s\ksum D^s}}|\ka-\ka'|\leq\de t\con[2]{S}\spr{D^s}^2,
\end{equation*}
where $\con[2]{S}$ is the condition number, with respect to the spectral norm, of the similarity matrix $S$ between $D^s\ksum D^s$ and its diagonal form. The estimate worsens for a very large condition number $\con[2]{S}$, which can be the case when eigenvectors of $D^s\ksum D^s$ are nearly linearly dependent. Assuming that this is not the case, the Bauer--Fike theorem ensures that all eigenvalues of $L^{s\kdot s}_{\de t}$ are contained within disks centred around eigenvalues of $D^s\ksum D^s$ and having radius $C\de t\ll\De t$. Therefore, to guarantee the stability of the extrapolated (slow) marginal variance, it suffices to bring all these disks inside $\dsk(-1/2,1/2)$. This is achieved whenever $\De t$ is such that $\spm{\De tD^s}\subset\dsk(-1/2,1/2-C\de t)$. According to this property, we can refer to the deflated disk $\dsk(-1/2,1/2-C\de t)$ as a \emph{sufficient stability region} of the micro-macro acceleration algorithm with extrapolation of (slow) marginal mean and variance.

\begin{figure}
\centering
\includegraphics[width=0.4\textwidth]{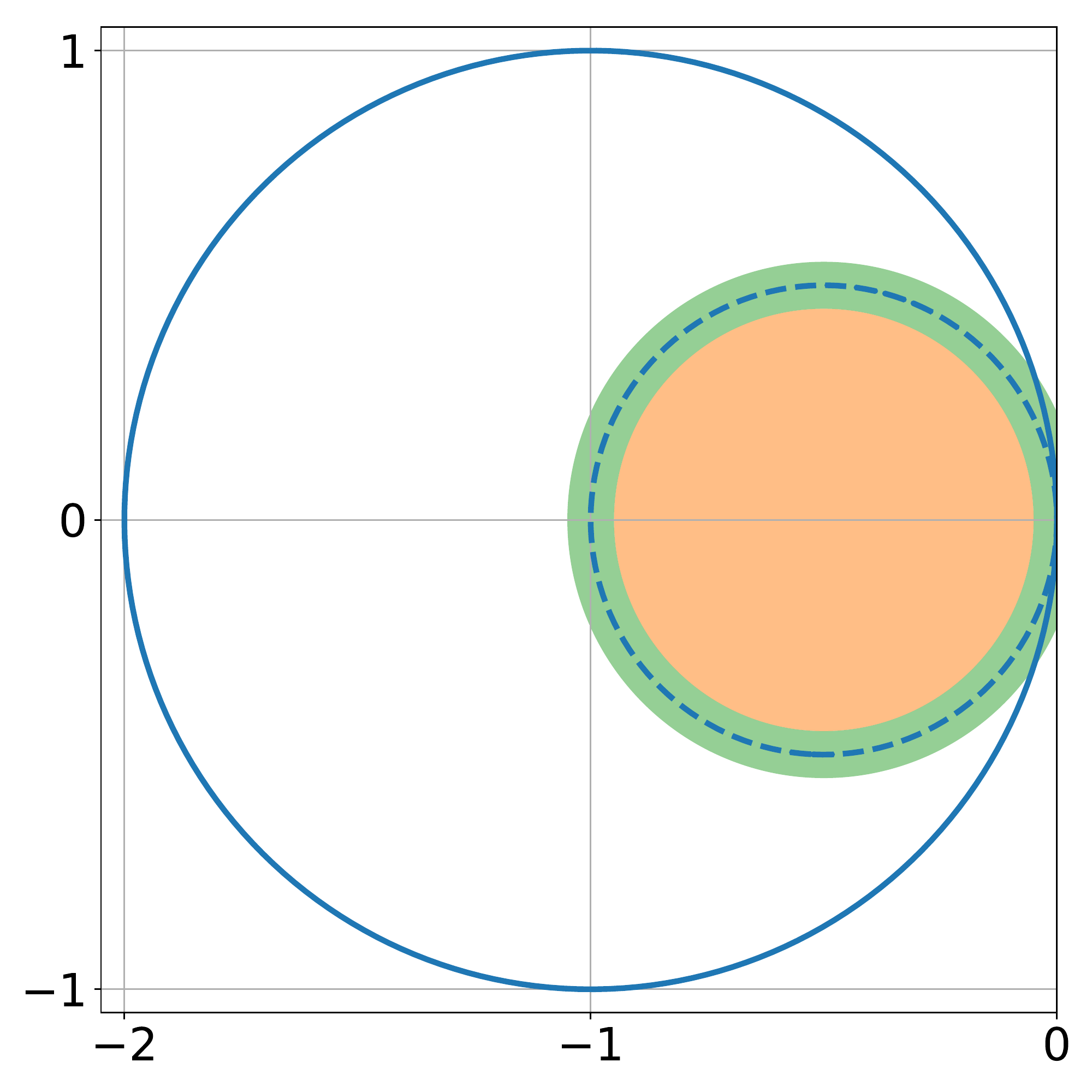}
\caption{
The stability regions for: the extrapolation of slow mean only (interior of solid blue circle), the extrapolation of slow mean and variance (union of orange disk and green ring), the ODE~\eqref{eq:ODEclos} of slow mean and variance evolution (interior of dashed blue circle). The stability region for the extrapolation of the slow mean overlaps with the corresponding region for the mean of ODE, whereas the stability region for the extrapolation of the slow mean and variance is a perturbation (on the order of micro time step $\de t$) of the corresponding region for the mean and variance of the ODE.
}
\label{fig:stability_region}
\end{figure}

In practice, as the value of $C\de t$ is small compared to $\De t$, one should choose a sufficient stability region as a target for the slow eigenvalues. Nevertheless, in particular cases (such as in Section~\ref{se:mM_gauss_smenvar_num}), the location of the eigenvalues of $L^{s\kdot s}_{\de t}$ within the disks can be more advantageous. In the most extreme situation, it may suffice to bring the centers of disks within a $C\de t$-neighbourhood of $\dsk(-1/2,1/2)$. This means that the stability of micro-macro acceleration method actually improves on the stability of deterministic closed ODEs for mean and variance, and we can refer to the inflated disk $\dsk(-1/2,1/2+C\de t)$ as \emph{necessary stability region}.

We summarize the relations between the different stability regions in Figure~\ref{fig:stability_region}. The necessary stability region for extrapolation of the slow variance (union of orange disk and green ring) lies mostly in the stability region for the extrapolation of the slow mean (interior of solid circle) but can exceed the stability region for variance closure (interior of dashed circle). The innermost orange disk represents the region where the stability is guaranteed by the Bauer--Fike theorem, whilst for slow modes in the green ring the stability hinges on the particular perturbation.

\subsubsection{Numerical illustrations}\label{se:mM_gauss_smenvar_num}
Let us now revisit system~\eqref{eq:test_sf}. In this case, we have just one slow mode and
\begin{equation*}
L^{s\kdot s}_{\de t} = 2a_{11} + \de ta^2_{11}.
\end{equation*}
The value of $a_{11}$ is negative, to comply with Assumption~\ref{as:test_mscale}, and we can read explicitly the  structure of perturbation in $L^{s\kdot s}_{\de t}$. The Kronecker's sum equals to $2a_{11}$ and the product is $a^2_{11}$. Since $\de ta^2_{11}>0$, the slow mode $L^{s\kdot s}_{\de t}$ will be moved to the right compared to the deterministic mode $2a_{11}$. Therefore, the stability bound for the extrapolation time step $\De t$ increases and reads
\begin{equation*}
\De t < \frac{2}{2a_{11} + \de ta^2_{11}}.
\end{equation*}
Moreover, within this bound, the asymptotic (slow) variance equals $-1/(2a_{11}+\de ta_{11}^2)$.

\textbf{Testing stability bounds for slow-fast system.}
For the first numerical illustration, we take the drift and diffusion matrices as in~\eqref{eq:sf_num}. The stability threshold for the extrapolation of slow mean and variance equals $2/1.9\approx 1.05$. We test this threshold by performing the micro-macro simulation as described in Section~\ref{se:mM_gauss_smean_num}, but with extrapolation of slow mean and variance. We use the adaptive time-stepping strategy and look at the statistics of actual time step to locate the threshold.

In Figure~\ref{fig:mMadap_linSF_smeanvar_conv} we plot the distances (in the Frobenius norm) of the vector mean $\mu$ and variance matrix $\Si$ of the micro-macro acceleration method to the mean and variance of the invariant distribution of the inner Euler-Maruyama scheme, equal to $0$ and $V_\infty^{\de t}$ respectively. The results are similar to the results for the extrapolation of mean only: we see that the adaptive strategy activates only after crossing $1.05$, which indicates instability as predicted by analytical considerations. In this case, the instabilities arise only in the evolution of variance, since the time steps considered in the experiments are well below the stability threshold, equal to $2$, for the extrapolation of the slow mean. Additionally, for stable extrapolation times, the distances between the first two moments of the micro-macro acceleration method and the invariant Gaussian distribution become small quickly. Since the intermediate distributions of the method are also Gaussian, the agreement between the first two moments indicates that these distributions converge, as time grows to infinity, to the invariant normal distribution $\nd_{0,V_\infty^{\de t}}$ in Kullback-Leibler divergence, as was the case in Section~\ref{se:mM_gauss_smean_ana}.

\begin{figure}
\centering
\includegraphics[width=0.8\textwidth]{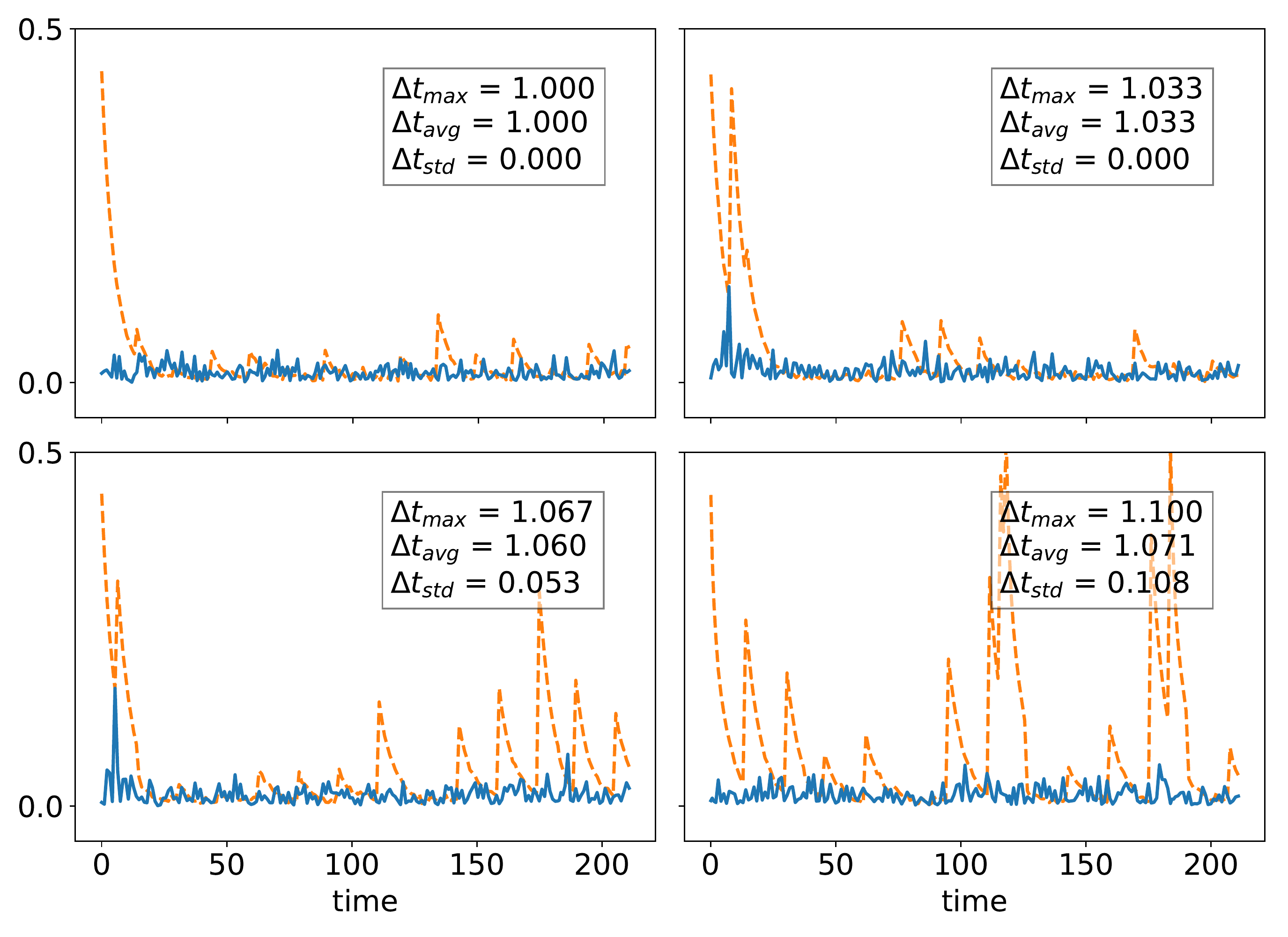}
\caption{The evolution of $\|\mu\|$ (solid blue) and $\|\Si - V_{\infty}^{\de t}\|$ (dashed orange), where $\mu$ is the vector mean and $\Si$ the variance matrix of the micro-macro acceleration method with adaptive extrapolation strategy of the slow mean and variance. The adaptive time-stepping activates after the maximal extrapolation time step $\De t_{\mrm{max}}$ crosses the theoretical value equal to $2/1.9 \approx 1.05$ (two bottom figures). The simulation is stable below this threshold, even though we crossed the deterministic stability bound equal to $1$ (two upper figures). Moreover, in the stable regime, both norms quickly become small, which indicates the convergence of full distributions to the invariant Gaussian density $\nd_{0,V_\infty^{\de t}}$.
}
\label{fig:mMadap_linSF_smeanvar_conv}
\end{figure}

\textbf{Increasing stability with additional micro steps}
In a second experiment, we examine the influence of the micro simulation with different numbers $K$ of micro time steps $\de t$, while keeping the microscopic time window $K\de t$ constant. To this end, we perform micro-macro simulation of~\eqref{eq:sf_num} with the adaptive extrapolation of the slow mean and variance and plot the deviation (in the Frobenius norm) of the evolving covariance matrix from the variance of the invariant distribution for different macroscopic time steps $\De t$. In Figure~\ref{fig:mM_stab_dt_vs_Dt}, we demonstrate that for slow-fast systems smaller time step (K=1 there) increases the extrapolation stability threshold, which is due to the increase in accuracy of the micro simulation. Here, we make micro simulation more accurate by fixing the value of the microscopic time window and subdividing it into smaller time steps.

In Figure~\ref{fig:mMadap_linSF_smeanvar_Kconv}, we plot the resulting curves for various combinations of $K$ and $\de t$ related by $K\de t = 0.12$. In the left plot, all three curves stay close to zero, implying that all simulations are stable. When we cross $\De t =0.95$, the simulation with $K=1$ micro time step becomes clearly unstable, while the two other do not develop large peaks for both $\De t =1.05$ and $\De t=1.15$. However, we observe that the simulations with $K=2,3$ perform similarly, and we do not get further increase in stability when taking even larger $K$. This could be compared to the right plot of Figure~\ref{fig:mM_stab_dt_vs_Dt} where we can see that as $\de t$ gets smaller, the increase in stability flattens.

\begin{figure}
\centering
\includegraphics[width=\textwidth]{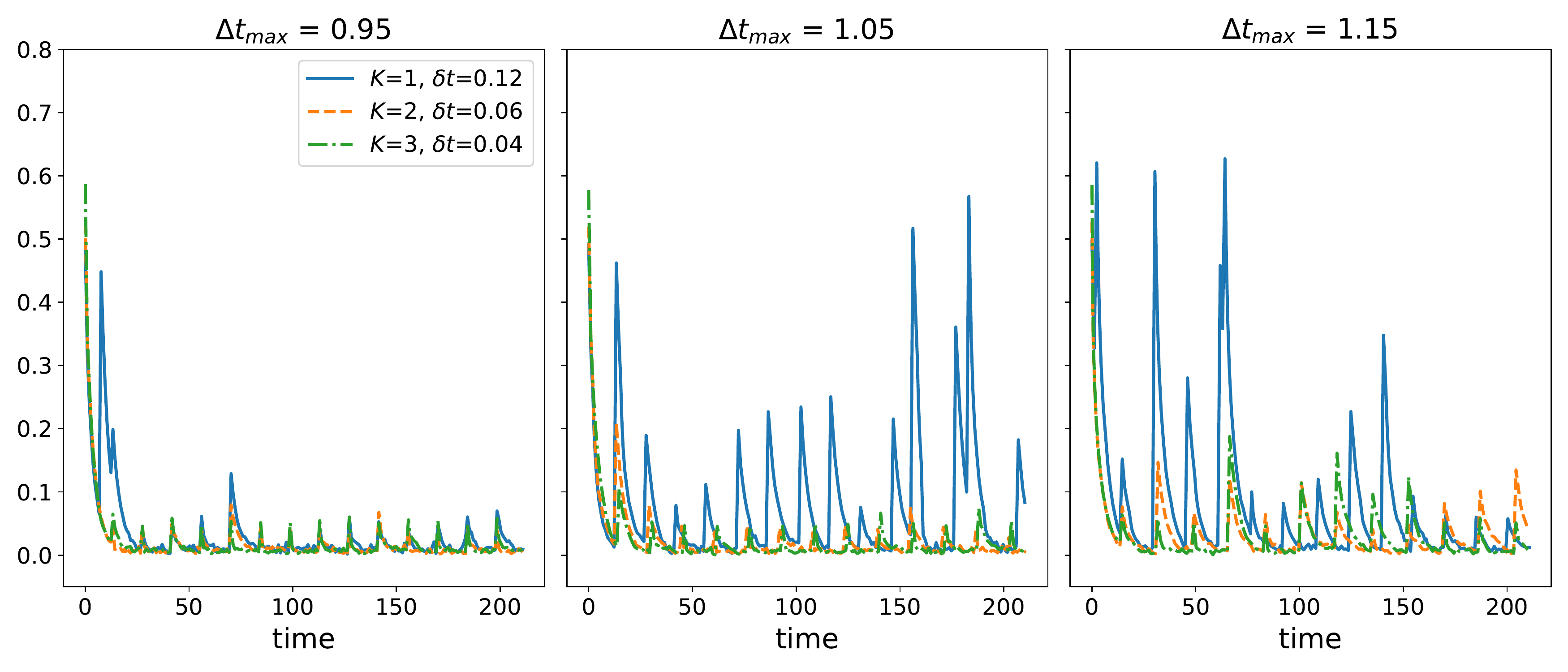}
\caption{The evolution of $\|\Sigma - V^{\delta t}_\infty\|$ for different macro time steps $\Delta t$, where $\Sigma$ is the full variance evolving under the micro-macro acceleration method with adaptive extrapolation strategy of the slow mean and variance, and $V^{\delta t}_\infty$ the variance of the invariant distribution of the Euler-Maruyama scheme with micro time step $\delta t$. The numbers of micro steps $K$ and corresponding micro time steps $\delta t$ are chosen so that the microscopic window $K\delta t$, on which Euler-Maruyama simulation of the full system is performed, stays constant and equal to $0.12$. With $K=1$ (solid blue line) the simulation is stable only in the left plot; for $\Delta t = 1.05, 1.15$ the large peaks appear revealing formation of instabilities. Increasing $K$ (dashed orange and dot-dashed green lines) yields more accurate microscopic simulations, which in turn can be seen to produce larger stability thresholds for $\Delta t$ (central and right plot). However, the boost in stability saturates quickly for linear systems under consideration and there is not much difference between $K=2$ and $K=3$.}
\label{fig:mMadap_linSF_smeanvar_Kconv}
\end{figure}

\subsubsection{ Convergence with non-Gaussian initial condition}\label{se:mM_nongauss_smenvar_num}
Hitherto, we restricted both the theoretical and numerical analysis of the micro-macro simulations with Gaussian initial conditions. Extensive study of the stability with non-Gaussian initial conditions is beyond the scope of this manuscript and we only present one numerical illustration -- with initial condition given as a Gaussian mixture (see also Remark~\ref{re:Gmix}). For some analytical results in the non-Gaussian case we refer to~\cite{Zielinski2019}.

In Figure~\ref{fig:mMadap_nongauss_convlinSF10}, we assess the quality of distributions evolving under the micro-macro acceleration method for~\eqref{eq:sf_num} by computing the Kullback-Leibler divergence~\eqref{eq:kld} of these distributions to the invariant distribution $\nd_{0,V^{\de t}_{\infty}}$. We use adaptive extrapolation of the slow mean and variance and take $\De t = 0.5$ and $\de t = 0.09$. As can be seen in the inset, we use a Gaussian mixture yielding a bimodal initial distribution. For reference, we choose the Gaussian initial condition such that its variance is equal to the variance of the mixture and its mean is such that the Kullback-Leibler divergence between the Gaussian and the invariant distribution is the same as between the mixture and the invariant distribution (this ensures that the curves start at the same point). There is no apparent difference in the transient regime between the mixture and Gaussian reference, which suggests that for linear systems the non-Gaussianity during the transient phase may not influence the stability results obtained here.

\begin{figure}
\centering
\includegraphics[width=0.8\textwidth]{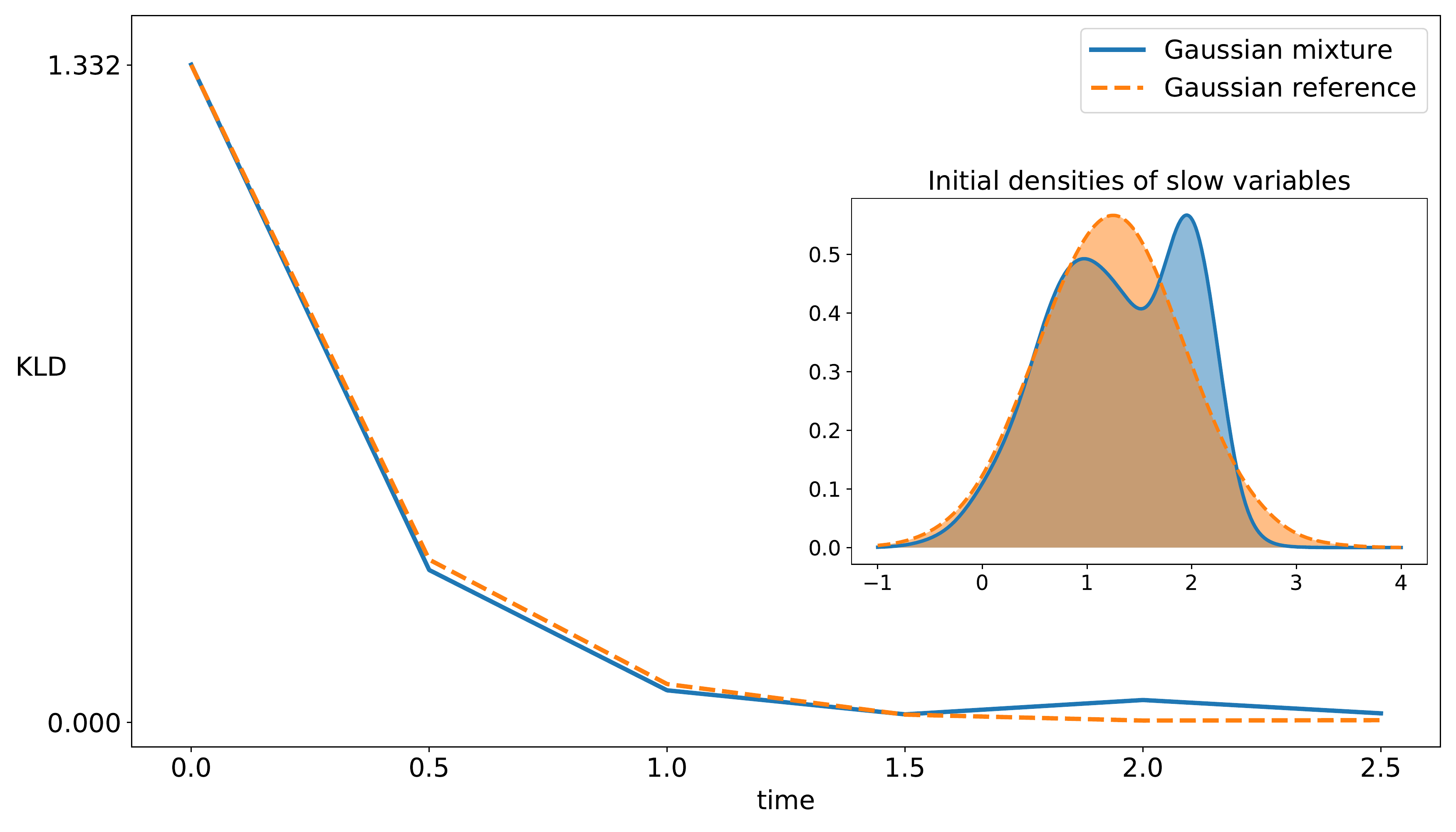}
\caption{Evolution of the Kullback-Leibler divergence between the slow invariant density $N_{0, V_{\delta t}}$ and the consecutive distributions of the micro-macro acceleration method for slow-fast system~\eqref{eq:sf_num} initialized with: a Gaussian mixture (solid blue line), and a reference Gaussian density (dashed orange line). The  slow reference Gaussian density is chosen so that it has the same variance and the same Kullback-Leibler divergence from $N_{0, V_{\delta t}}$ as the Gaussian mixture (plotted in the inset). For both initial distributions, the method converges to the invariant density and we do not observe any deviation in the transient region due to non-Gaussianity.}
\label{fig:mMadap_nongauss_convlinSF10}
\end{figure}

\section*{Conclusions and outlook}
For our theoretical analysis, we inquired the linear stability of a micro-macro method in a setting that combines short propagation of the exact distribution with forward in time extrapolation of a few macroscopic state variables. For this study, we concentrated on linear vector equations, that can accommodate multiple scales, and distinguished a particular subclass (in Assumption~\ref{as:test_mscale}) of linear systems. We demonstrate in Theorem~\ref{th:mM_gauss_smean} how having two concurrent time steps in the method --~microscopic for path simulation and macroscopic for extrapolation~-- allows to bypass the microscopic stability constraints on the macroscopic one. While the macroscopic stability bound for the extrapolation of slow mean coincides exactly with the stability threshold of the ODE closure for the slow mean, we showed that in the case of slow mean and variance extrapolation, it corresponds to the closed ODE for slow mean and variance with perturbation proportional to the microscopic time step.

For numerical experiments, though performed in the Gaussian setting, we discretize in probability and employ short bursts of path simulations to mimic the situation where the exact distributions cannot be followed and compare it to our analytical results. First, we demonstrated the connection between crossing the macroscopic stability threshold and lack of convergence of the Newton-Raphson iteration to compute the matching, which gives a useful benchmark -- matching failure -- to spot instabilities developing in the micro-macro acceleration method (Figure~\ref{fig:mMnoad_mean_nostab}). Using this benchmark, as illustrated in Figure~\ref{fig:mM_stab_dt_vs_Dt} (left), we observe a complete agreement of numerical stability with analytical bounds in the drift block-diagonal case. However, there is a qualitative difference when applying the micro-macro acceleration method with slow mean extrapolation to the slow-fast systems, Figure~\ref{fig:mM_stab_dt_vs_Dt} (right). To what extent we can use our analysis in the framework of slow-fast system remains for future work.
We also illustrated the relevance of an adaptive procedure -- based on matching failures -- for the extrapolation time step (Figure~\ref{fig:mMadap_smean_nostab}), which allows to perform micro-macro acceleration when the stability threshold remains unknown. For the extrapolation of slow mean and variance, we also observe agreement between theoretical findings and numerical simulation with empirical measures, exemplified in Figure~\ref{fig:mMadap_linSF_smeanvar_conv}.

\subsection*{Acknowledgements}
GS and PZ acknowledge the support of the Research Council of the University of Leuven through grant \lq PDEOPT\rq, and of the Research Foundation -- Flanders (FWO -- Vlaanderen) under grant G.A003.13.

\appendix

\section{Monte-Carlo simulation}\label{se:MC_sim}
Let us discuss in more detail the discretisation in probability of the micro-macro acceleration algorithm presented in Sections~\ref{se:mM_accel_alg} and~\ref{se:mM_accel_alg_linSDE}. We consider an initial distribution given as a random empirical measure
\begin{equation*}
P^{\De t,J}_{n}=\sum_{j=1}^{J}w^j_n\de_{X^j_n},
\end{equation*}
where $X^j_n$, $j=1,\dotsc,J$, are i.i.d.~replicas with associated weights $w^j_n$.

\textit{Stage~1: Propagation of microscopic laws.}
In this stage, we freeze the weights and propagate each replica over the $K$ steps of inner integrator $\prop^{\de t}$ (compare with~\eqref{eq:em_test})
\begin{equation*}
X^{\de t, j}_{n,k+1} = \prop^{\de t}(a, b, X^{\de t, j}_{n,k},\xi^{\de t, j}),\quad X^{\de t,j}_{n,0}=X^j_n,
\end{equation*}
where $\xi^{\de t, j}$ are $J$ i.i.d.~centered normal variables with variance $\de t$ (replicas of $\de W_{k+1}$). In particular, when applying the Euler-Maruyama method to~\eqref{eq:genSDE} we get
\begin{equation*}
X^{\de t, j}_{n,k+1} = X^{\de t, j}_{n,k} + a\big(X^{\de t, j}_{n,k}\big)\de t + b\big(X^{\de t, j}_{n,k}\big)\xi^{\de t, j}.
\end{equation*}
The replicas $X^{\de t, j}_{n,k}$ are associated to the following empirical measures
\begin{equation*}
P^{\de t,J}_{n,k} = \sum_{j=1}^{J}w^j_n\de_{X^{\de t,j}_{n,k}},\quad k=1,\dotsc,K.
\end{equation*}

\textit{Stage~2: Restriction to a finite number of observables.}
In this stage, we evaluate the restriction operator $\res$ on the empirical measures to obtain $K$ empirical observables as $\mbf{m}^{J}_{n,k} = \res(P^{\de t,J}_{n,k})$. For example, when restricting to the slow mean only as in~\eqref{eq:res_smean}, we find $K$ vectors
\begin{equation*}
\mbf{m}^{s,J}_{n,k} = \sum_{j=1}^{J}w^j_nY^{\de t,j}_{n,k},
\end{equation*}
where $Y^{\de t,j}_{n,k}=\proj^sX^{\de t,j}_{n,k}$ and $\proj^s$ is the projection onto the slow variables.

\textit{Stage~3: Extrapolation of macroscopic states.}
We apply the extrapolation operator $\ext^{\De t - K\de t}$ to the empirical macroscopic states $\mbf{m}^{J}_{n,0},\dotsc,\mbf{m}^{J}_{n,K}$ and obtain $\mbf{m}^{s,J}_{n+1}$ as in formula~\eqref{eq:extrap}.
In the special case of linear extrapolation one proceeds according to~\eqref{eq:lin_extrap}, with $\mbf{m}^{s,J}_{n,0}$ and $\mbf{m}^{s,J}_{n,K}$ plugged in the right-hand side.

\textit{Stage~4: Matching.}
Finally, the matching amounts to re-weighting the replicas $X^{\de t, j}_{n,K}$ using the Lagrange multipliers associated to procedure~\ref{eq:match_smean} or~\ref{eq:match_smeanvar} on page~\pageref{eq:match_smean}. Concentrating on~\ref{eq:match_smean} only, we approximate the Lagrange multipliers by applying the Newton-Raphson iteration to equation (compare with~\eqref{eq:lagr_mean})
\begin{equation}\label{eq:lagr_MC}
\grad[\la]A^s\big(\la,P^{\de t,J}_{n,K}\big)=\mbf{m}^{s,J}_{n+1},\quad\text{where}\quad A^s\big(\la,P^{\de t,J}_{n,K}\big) = \ln\sum_{j=1}^{J}w^j_n\exp\big(\la\sdot\proj^sX^{\de t,j}_{n,K}\big).
\end{equation}
In this fashion, we obtain, up to a given tolerance, the vector $\oline{\la}^{s,J}_{n+1}$ of Lagrange multipliers with which, following~\eqref{eq:dist_mean}, we evaluate weights
\begin{equation*}
w^j_{n+1} = w^j_n\exp\Big(\oline{\la}^{s,J}_{n+1}\sdot\proj^sX^{\de t,j}_{n,K} - A^s\big(\oline{\la}^{s,J}_{n+1},P^{\de t,J}_{n,K}\big)\Big).
\end{equation*}
The matched empirical distribution reads
\begin{equation*}
P^{\De t,J}_{n+1} = \sum_{j=1}^{J}w^j_{n+1}\de_{X^{\de t,j}_{n,K}}.
\end{equation*}

\section{Gap in the drift spectrum and the time scales of~\eqref{eq:test}}\label{se:tscale_dynsys}
To elucidate how Assumption~\ref{as:test_mscale} influences the time scales present in the stochastic dynamics, recall first that equation~\eqref{eq:test} is related to the \emph{Ornstein-Uhlenbeck} operator
\begin{equation}\label{eq:ou_op}
\gen = Ax\sdot\grad[x] + \frac{1}{2}\tr{B\hess[x]}.
\end{equation}
SDE~\eqref{eq:test} and operator~\eqref{eq:ou_op} are connected through the Markov semigroup $(e^{t\gen})_{t\geq0}$, generated by $\gen$, that satisfy
\begin{equation*}
\big(e^{t\gen}f\big)(x) = \Exp[][f(X_t)|X_t=x],
\end{equation*}
for every $t\geq0$ and $f\in\contb(\R^d)$, the space of all continuous and bounded functions on $\R^d$. The assumptions that $\spm{A}\subset\C_-$ and $B$ is positive definite ensure the existence of a unique Gaussian invariant measure $\nd_{0,V_\infty}$ for $(e^{t\gen})_{t\geq0}$, where $V_\infty$ is given by~\eqref{eq:test_equivar}. To be more precise, the condition for invariance reads
\begin{equation}\label{eq:inv_meas}
\Exp[\nd_{0,V_\infty}]\big[e^{t\gen}f\big] = \Exp[\nd_{0,V_\infty}][f],
\end{equation}
for all $f\in\contb(\R^d)$. Moreover, the semigroup $(e^{t\gen})_{t\geq0}$ extends to a strongly continuous semigroup of positive contractions in the complex Hilbert space $\leb{2}_{\C}(\R^d, d\nd_{0,V_\infty})$~\cite{Lunardi1997}.

The time scales induced by the semigroup $(e^{t\gen})_{t\geq0}$ in the space $\leb{2}_{\C}(\R^d, d\nd_{0,V_\infty})$ are determined by the eigenvalue problem $\gen f = \ga f$~\cite[p.~371]{ZhaHarSch2016}. Every eigenpair $(\ga,\ph)$, with $\re(\ga)<0$ and $\|\ph\|_2=1$, is related to a decay of $e^{t\gen}\ph$ towards the equilibrium on time scales of order $|2\ga|^{-1}$, to wit
\begin{equation*}
\Var[\!\nd_{0,V_\infty}]\big[e^{t\gen}\ph\big] = \big\|e^{t\gen}\ph - \Exp[\nd_{0,V_\infty}][\ph]\big\|_2 = e^{2\re(\ga) t},
\end{equation*}
where we used condition~\eqref{eq:inv_meas} and $\|\cdot\|_2$ denotes the associated $\leb{2}_{\C}$-norm. Having a complete orthonormal system $\{\ph_p\}_{p=1,2\dotsc}$ of eigenfunctions in $\leb{2}_{\C}(\R^d, d\nd_{0,V_\infty})$, the Fourier expansion
\begin{equation}\label{eq:fourier}
f = \sum_{p=1}^{\infty}\bra f,\ph_p\ket_2\,\ph_p
\end{equation}
decomposes the trend of $e^{t\gen}f$ towards the equilibrium into separate modes, supported by the invariant subspaces generated by each $\ph_p$ and exponentially decaying with rates given by $\re(\ga_p)$.

The spectrum of the Ornstein-Uhlenbeck operator~\eqref{eq:ou_op} in $\leb{2}_{\C}(\R^d, d\nd_{0,V_\infty})$ consists of all complex numbers $\ga = \sum_{i=1}^{d}n_i\ka_i$ with $n_i\in\N_0, \ka_i\in\spm{A}$, and all the (generalised) eigenfunctions are polynomials and form a complete system~\cite{MetPalPrio2002}. Moreover, the invariant subspaces related to each $\ga$ consist of homogeneous polynomials in the variables induced by the spectral decomposition of $A$ and with degrees ranging throughout all $n_i>0$ that appear in the sums generating $\ga$~\cite[Sect.~4]{MetPalPrio2002}.

There are three main implications of these facts in our setting. First, the spectrum of $\gen$ is independent of the diffusion matrix $B$. This demonstrates that $B$ has no effect on the time scales of the dynamics and justifies the omission of any assumptions on its spectrum. Second, the eigenvalues of the drift matrix are embedded inside the spectrum of $\gen$ and induce the most prominent time scales. Indeed, every $\ga=\ka_i$ is uniquely determined by $n_i=1$ and $n_j=0$ for $j\neq i$. The associated eigenfunction is a homogeneous polynomial of degree $1$ in a variable associated to the invariant subspace of $\ka_i$. Therefore, in the Fourier expansion~\eqref{eq:fourier}, all eigenvalues $\ka_i$ constitute the first approximation of $f$. Finally, the dynamics of $e^{t\gen}$ has infinite number of different time scales but the gap in $\spm{A}$ reveals itself at the lowest order modes. This can be seen by applying the decomposition of the state space $\R^d$ into the slow vector variable $y$, associated to $\Om^s$, and the fast vector variable $z$, associated to $\Om^f$, see the paragraph following Assumption~\ref{as:test_mscale}. This grouping decomposes the first approximation of $f$ into polynomials in $y$ and $z$ that equilibrate under the action of $e^{t\gen}$ on two different time scales with gap given by the gap in $\spm{A}$.

\printbibliography

\end{document}